\documentclass[12pt]{article}
\usepackage{amsfonts,amsmath,amsbsy,amssymb,amsthm}

\newtheorem*{linflips}{Main Theorem 1A:  Classification of Linear Flips}
\newtheorem*{slinflips}{Main Theorem 1B:  Classification of $\sigma$-Semilinear Flips}
\newtheorem*{lft}{Main Theorem 2:  Linear Flag Transitivity}
\newtheorem*{slft}{Main Theorem 3:  $\sigma$-Semilinear Flag Transitivity}
\newtheorem*{cflips}{Main Theorem 1:  Classification of Flips}
\newenvironment{rem}{\noindent\emph{Remark.  }}{\hfill$\lozenge$}
\newtheorem{thm}{Theorem}[section]
\newtheorem{lem}[thm]{Lemma}
\newtheorem{cor}[thm]{Corollary}
\theoremstyle{definition}
\newtheorem{defn}{Definition}[section]
\newtheorem{cons}{Construction}
\newtheorem{note}{Note}[section]
\newcommand{\dperp}{\perp\hspace{-.15cm}\perp}

\begin{document}

\title{On Flips of Unitary Buildings I:  Classification of Flips}

\maketitle

\begin{center}
{\large{Rieuwert J. Blok$^{1}$ and Benjamin Carr$^{1,2}$\\{\large{$^{1}$Department of Mathematics and Statistics\\Bowling Green State University\\ Bowling Green, Ohio 43402 USA\\ \vspace{3mm}$^{2}$MCPE Division\\Lindenwood University\\209 S. Kingshighway\\St. Charles, Missouri 63301 USA\\bcarr@member.ams.org}}}}
\end{center}


\vspace{.5cm}


\vspace{.5cm}

{\large{Suggested Running Title:  On Flips of Unitary Buildings I}}

\begin{abstract}
We classify flips of buildings arising from non-degenerate unitary spaces of dimension at least 4 over finite fields of odd characteristic in terms of their action on the underlying vector space.  We also construct certain geometries related to flips and prove that these geometries are flag transitive.
\end{abstract}

\newpage

\section{Introduction}

\subsection{History}

This paper should be viewed as part of a program described in \cite{BGHS2001} to prove theorems similar to Phan's theorem.  These so-called ``Phan-type" theorems have been studied in a number of papers (e.g. \cite{BCS2004}, \cite{BGHS2007}, \cite{GHNS2005}) initially in order to aid the Gorenstein-Lyons-Solomon revision of the proof of the Classification of Finite Simple Groups.  Roughtly speaking, these ``Phan-type'' theorems allow for the recognition of a group based on amalgams of subgroups that are produced by the group acting on a geometry.  These results all rely on the fact that if a geometry is simply connected, then a flag transitive automorphism group of the geometry is the universal completion of its amalgam of maximal parabolic subgroups.  The reader interested in more detail should consult \cite{BGHS2001} for an overview.

The strategy to prove further Phan-type theorems is to identify a simply connected flag transitive geometry, and a group acting flag transitively on the geometry.  The notion of a \textbf{flip} (or \textbf{Phan involution} to some authors) was introduced in \cite{BGHS2001} as a means to produce new geometries which are, in many cases, simply connected and flag transitive.

Flips are studied in a more general context in \cite{Horn2009} and \cite{GHM2009} where their properties are explored, however the authors do not make a closer study of flips of the building under consideration here.

\subsection{The Results of This Paper}

Throughout this paper, $q$ denotes an odd prime power, $\Delta$ denotes the building associated to the geometry of totally isotropic subspaces of a $2n$-dimensional ($n\geq 2$) non-degenerate unitary space $(V,\beta)$ over $\mathbb{F}=\mathbb{F}_{q^{2}}$, and $\sigma$ denotes the $q$th power map on $\mathbb{F}$.

In this paper we classify flips of $\Delta$ in terms of their action on $(V,\beta)$.  Since the interest in flips arises because of the possibility of proving further Phan-type theorems our proof is highly geometric, relying on the construction of geometries induced by the flip.  Finally, we prove that these geometries are flag transitive and therefore can be used to prove Phan-type theorems when they are simply connected.  In \cite{Carr2009b} we study the topological properties of these geometries and show that in large rank they are simply connected.  

The main results of this paper are as follows:

\begin{cflips}\label{thm.cflips}  Let $\varphi$ be a flip of $\Delta$.  Then $\varphi$ is induced by a semilinear transformation $f$ of the underlying unitary space $V$ such that exactly one of the following holds:
\begin{enumerate}
\item[(i)]  $f$ is a linear isometry of $(V,\beta)$, $f^{2}=\mathrm{id}$ on $V$, and there is a hyperbolic basis $\{e_{i},f_{i}\}_{i=1}^{n}$ for $V$ such that $f(e_{i})=f_{i}$ for $i=1,\ldots,n$;
\item[(ii)]  $f$ is a linear anti-isometry of $(V,\beta)$, $f^{2}=\mathrm{id}$ on $V$, and there is a hyperbolic basis $\{e_{i},f_{i}\}_{i=1}^{n}$ for $V$ such that $f(e_{i})=\alpha f_{i}$ and $f(f_{i})=\alpha^{-1} e_{i}$ for $i=1,\ldots,n$, where $\alpha$ is a trace 0 element of $\mathbb{F}$;
\item[(iii)]  $f$ is a $\sigma$-semilinear isometry of $(V,\beta)$, $f^{2}=\mathrm{id}$ on $V$, and there is a hyperbolic basis $\{e_{i},f_{i}\}_{i=1}^{n}$ for $V$ such that $f(e_{i})=f_{i}$ for $i=1,\ldots,n$;
\item[(iv)]  $f$ is a $\sigma$-semilinear isometry of $(V,\beta)$, $f^{2}=\mathrm{id}$ on $V$, and there is a hyperbolic basis $\{e_{i},f_{i}\}_{i=1}^{n}$ for $V$ such that for $i=1,\ldots,n-1$, $f(e_{i})=f_{i}$, $f(f_{i})=e_{i}$ and there is a non-square $\lambda\in\mathbb{F}$ with $f(e_{n})=\lambda f_{n}$ and $f(f_{n})=\sigma(\lambda^{-1})e_{n}$.
\end{enumerate}
Conversely any semilinear transformation of $V$ satisfying one of (i)-(iv) induces a flip of $\Delta$.
\end{cflips}

So there are up to a unitary base change only four flips of $\Delta$.  Each flip gives rise to non-isomorphic geometries which can be used to prove Phan-type theorems about flag-transitive automorphism groups of the geometries when the geometries are simply connected by appealing to Tits' Lemma (Corollaire 1 of \cite{Tits1986}.)

In the body of the paper this theorem is split into four pieces.  First we prove in Lemma \ref{lem.induceflip} that every flip of $\Delta$ is induced by some linear isometry, linear anti-isometry, or $\sigma$-semilinear isometry of $(V,\beta)$.  Then, in Lemma \ref{lem.converse} we prove that a semilinear transformation of $V$ satisfying any of (i)-(iv) induces a flip of $\Delta$.  We then prove in Main Theorem 1A (Section \ref{sec.linflips}) that if the transformation is linear, then (i) or (ii) holds.  Finally in Main Theorem 1B (Section \ref{sec.slinflips}) we prove that if the transformation is $\sigma$-semilinear then either (iii) or (iv) holds.

In addition to classifying the flips of $\Delta$ we prove the following results regarding the geometries $\Gamma(n,q)$ and $\Gamma_{1}(n,q)$.  The construction of $\Gamma(n,q)$ is carried out in Section \ref{sec.geom} and the construction of $\Gamma_{1}(n,q)$ is carried out in Section \ref{sec.sgeom}.

\begin{lft}  If $\varphi$ is a flip of $\Delta$ induced by a linear transformation of $(V,\beta)$ then the geometry $\Gamma(n,q)$ is flag transitive.
\end{lft}

\begin{slft}  If $\varphi$ is a flip of $\Delta$ induced by a $\sigma$-semilinear transformation of $(V,\beta)$ then the geometry $\Gamma_{1}(n,q)$ is flag transitive.
\end{slft}

With these results in hand, the last step to establishing new Phan-type theorems is to study the homotopy properties of these geometries.  This is done for large rank cases in \cite{Carr2009b}.

\subsection{Acknowledgments}

The results of this paper are part of the second authors Ph.D. thesis, \cite{CarrD}, under the supervision of the first author.  We would also like to express our gratitude to Professor Antonio Pasini for a careful proofreading of the paper and his many helpful comments.  Finally we would like to thank the anonymous referee for his comments.

\section{Definitions}\label{ch.group}

\subsection{Incidence Geometry}

\begin{defn}  Let $I$ be a set.  A \textbf{pregeometry} over $I$ is a set $\Gamma$ together with a \textbf{type function} $t:\Gamma\to I$ and a symmetric, reflexive \textbf{incidence relation} $\sim$ on $\Gamma$ with the property that for $x$, $y\in \Gamma$, $x\sim y$ and $t(x)=t(y)$ implies $x=y$.  The set $I$ is called the \textbf{type set} of the pregeometry.  The cardinality of $I$ is called the \textbf{rank} of the pregeometry.  The elements of $\Gamma$ are called the \textbf{objects} of the pregeometry.

A pregeometry is often denoted by an ordered quadruple $(\Gamma,I,t,\sim)$.  If the context is unambiguous the pregeometry may be denoted $\Gamma$.
\end{defn}

\begin{defn}  Let $\Gamma$ be a pregeometry.  A \textbf{flag} is a set of pairwise incident elements.  The \textbf{type of a flag} $\mathcal{F}=\{F_{i_{1}},\ldots,F_{i_{k}}\}$ is $t(\mathcal{F})=\{t(F_{i_{j}})|j=1,\ldots,k\}$.  The \textbf{cotype} of $\mathcal{F}$ is $I\setminus t(\mathcal{F})$.  A flag of type $I$ is called a \textbf{chamber}.

A flag $F$ is \textbf{maximal} if it is not properly contained in any other flag.

$\Gamma$ is \textbf{transversal} if every maximal flag is a chamber.  A transversal pregeometry is called a \textbf{geometry}.
\end{defn}

\begin{defn}  Let $F$ be a flag in a geometry $\Gamma$.  The \textbf{residue} of $F$ in $\Gamma$, denoted $\mathrm{res}_{\Gamma}(F)$, is the set of all elements of $\Gamma\setminus F$ that are incident to all elements of $F$.  The residue of a flag is a geometry with type set $I\setminus t(F)$.  The rank of a residue is called the \textbf{corank} of the flag.
\end{defn}

\begin{defn}  An \textbf{automorphism} of a geometry $\Gamma$ is a permutation of its objects that preserves incidence and type.  Denote the group of all automorphisms of $\Gamma$ by $\mathrm{Aut}(\Gamma)$.
\end{defn}

\begin{defn}  Let $\Gamma$ be a geometry and let $G\leq\mathrm{Aut}(\Gamma)$.  We say that $G$ acts \textbf{flag transitively} on $\Gamma$ if, given two flags $C$, $D$ of $\Gamma$ of the same type, there is an element $g\in G$ so that $g(C)=D$.  If $\mathrm{Aut}(\Gamma)$ acts flag transitively on $\Gamma$ then $\Gamma$ is called a \textbf{flag transitive geometry}.
\end{defn}

\subsection{Buildings and Flips}

The material in this section follows \cite{AbBr2008} with the exception of the definition of a flip, which is taken from \cite{BGHS2001}.

\begin{defn}\label{def.newbuilding}  Let $(W,S)$ be a Coxeter system.  A \textbf{building} of type $(W,S)$ is a non-empty set $\mathcal{C}$ together with a map $\delta:\mathcal{C}\times\mathcal{C}\to W$ such that for all $C, D\in\mathcal{C}$ we have:
\begin{enumerate}
\item[(i)]  $\delta(C,D)=1$ if and only if $C=D$;
\item[(ii)]  If $\delta(C,D)=w$ and $C'\in\mathcal{C}$ with $\delta(C',C)=s\in S$ then $\delta(C',D)=sw$ or $w$.  Moreover if $l(sw)=l(w)+1$ then $\delta(C',D)=sw$.
\item[(iii)]  If $\delta(C,D)=w$ then for any $s\in S$ there is an element $C'\in \mathcal{C}$ with $\delta(C',C)=s$ and $\delta(C',D)=sw$.
\end{enumerate}
The elements of $\mathcal{C}$ are called \textbf{chambers}.

A building of type $(W,S)$ is called \textbf{spherical} if $(W,S)$ is a spherical Coxeter system.

Let $(\mathcal{C},\delta)$ be a spherical building of type $(W,S)$.  Two chambers $C$ and $D$ are \textbf{opposite} if $\delta(C,D)=w_{0}$, where $w_{0}$ is the longest word of $(W,S)$.
\end{defn}

\begin{defn}  Let $(\mathcal{C},\delta)$, $(\mathcal{C}',\delta')$ be buildings of type $(W,S)$.  An \textbf{isomorphism} between $(\mathcal{C},\delta)$ and $(\mathcal{C}',\delta')$ is a bijection $\rho:\mathcal{C}\to\mathcal{C}'$ such that for all $u$, $v\in \mathcal{C}$, $\delta(u,v)=\delta'(\rho(u),\rho(v))$.

An \textbf{automorphism} of $(\mathcal{C},\delta)$ is an automorphism of $(\mathcal{C},\delta)$ with itself.
\end{defn}

\begin{rem}  What we have called isomorphisms are sometimes called \textbf{isometries} of the building, with the term isomorphism reserved for a larger class of maps.  For the building associated to the geometry of totally isotropic subspaces of a non-degenerate unitary space over a finite field the two terms are equivalent.
\end{rem}

\begin{defn}  An \textbf{apartment} of a building $(\mathcal{C},\delta)$ of type $(W,S)$ is a subset $A$ of $(\mathcal{C},\delta)$ such that $(A,\delta|_{A})$ is isomorphic to the Coxeter building of type $(W,S)$.
\end{defn}

\begin{defn}  Let $(\mathcal{C},\delta)$ be a spherical building of type $(W,S)$ and let $w_{0}$ be the longest word of $(W,S)$.  A \textbf{flip} is a map $f:\mathcal{C}\to\mathcal{C}$ such that for all $C,D\in\mathcal{C}$:
\begin{enumerate}
\item[(i)]  $f^{2}(C)=C$;
\item[(ii)]  $\delta(C,D)=w_{0}\delta(f(C),f(D))w_{0}$;
\item[(iii)]  There exists $C\in\mathcal{C}$ such that $\delta(C,f(C))=w_{0}$.
\end{enumerate}
\end{defn}

\begin{note}  It follows from (ii) that a flip is an isometry of the building if and only if $w_{0}$ is central in $W$. In particular this holds for the building studied in this paper.
\end{note}

\subsection{The Apartments of $\Delta$}

Recall that $\Delta$ denotes the building associated to the geometry of totally isotropic subspaces of the unitary space $(V,\beta)$.

We now describe the apartments of $\Delta$.  This description is valid in a wider context, the interested reader can consult Chapter 7 of \cite{Taylor1992}. For concreteness we assume that $V$ is a left vector space over $\mathbb{F}$.  Because of this convention we have that for all $u$, $v\in V$, $\lambda$, $\mu\in \mathbb{F}$
$$
\beta(\lambda u,\mu v)=\lambda \beta(u,v)\sigma(\mu)=\lambda\sigma(\mu)\beta(u,v).
$$

\begin{cons}  Let $U_{1}=\langle e_{1},\ldots,e_{n}\rangle$ be a maximal totally isotropic subspace of $V$.  It can be shown (see for example Lemma 7.5 of \cite{Taylor1992}) that there is a totally isotropic subspace $U_{2}=\langle f_{1},\ldots,f_{n}\rangle$ such that $(e_{1},f_{1}),\ldots,(e_{n},f_{n})$ are pairwise orthogonal hyperbolic pairs, i.e. $\beta(e_{i},f_{j})=\delta_{ij}$, where $\delta_{ij}$ is the Kronecker $\delta$, and $\beta(e_{i},e_{j})=\beta(f_{i},f_{j})=0$ for all $i$, $j=1,\ldots,n$.  Recall that $\{e_{i},f_{i}\}_{i=1}^{n}$ forms a hyperbolic basis for $V$.  The \textbf{polar frame} associated to $\{U_{1},U_{2}\}$ is
$$
\mathcal{F}=\{\langle e_{i}\rangle,\langle f_{i}\rangle|1\leq i\leq n\}.
$$
Notice that it is the subspaces $\langle e_{i}\rangle$, $\langle f_{i}\rangle$ that define the polar frame, not the particular vectors $e_{i}$, $f_{i}$.  Hence different hyperbolic bases for $V$ may give rise to the same polar frame.

The apartment of $\mathcal{F}$ in $\Delta$ consists of all flags $F$ that are spanned by some subset of $\{e_{1},f_{1},\ldots,e_{n},f_{n}\}$.  This apartment is denoted $\Sigma(\mathcal{F})$.

Every apartment of $\Delta$ is of the form $\Sigma(\mathcal{F})$ for some polar frame $\mathcal{F}$.
\end{cons}

\begin{note}  In what follows, if $\mathcal{F}=\{\langle e_{i}\rangle,\langle f_{i}\rangle|i=1,\ldots,n\}$ is a polar frame we denote the apartment $\Sigma(\mathcal{F})$ by
$$
\Sigma(\mathcal{F})=\Sigma(\langle e_{i}\rangle,\langle f_{i}\rangle|i=1,\ldots,n)
$$
or if the context permits,
$$
\Sigma(\mathcal{F})=\Sigma(e_{i},f_{i}).
$$
This last notation is somewhat of an abuse, since the collection of pairwise orthogonal hyperbolic pairs is not uniquely determined by the polar frame, but if we start with this collection we know the frame, and hence the apartment.
\end{note}

Now that we know what the apartments of $\Sigma(\Gamma)$ look like, we can describe when two chambers are opposite.  The following appears as Exercise 9.16(ii) of \cite{Taylor1992}.

\begin{thm}\label{thm.oppositechambers}  Two chambers $C=(C_{i})_{i=1}^{h}$ and $D=(D_{i})_{i=1}^{h}$ in the building of a non-degenerate polar geometry $(W,\rho)$ of rank $h>0$ are opposite if and only if for all $i$,
$$
C_{i}^{\perp}\cap D_{i}=\{0\}.
$$
\end{thm}

\section{First Results on Flips}

\subsection{The Unitary Building and its Flips}

 Let $\Delta$ denote the building associated to the polar geometry of $(V,\beta)$.  It is shown in Chapter 7 of \cite{Tits1974} that $\Delta$ is a building of type $C_{n}$.

\begin{defn}  A \textbf{similitude} of the polar space $(V,\beta)$ is a $\tau$-semilinear transformation ($\tau\in\mathrm{Aut}(\mathbb{F})$) $f$ of $V$ with the property that there exists some $a=\sigma(a)\in\mathbb{F}_{q}$ such that for all $u$, $v\in V$, $\beta(f(u),f(v))=a\tau(\beta(u,v))$.  If $a=1$, $f$ is an \textbf{isometry}.  If $a=-1$ then $f$ is called an \textbf{anti-isometry}.

The group of all similitudes of $(V,\beta)$ is denoted $\Gamma\mathrm{U}(V)$.  By $\mathrm{P}\Gamma\mathrm{U}(V)$ we denote the quotient of $\Gamma\mathrm{U}(V)$ by its center.
\end{defn}

\begin{thm}\label{thm.autofbuilding}  $\mathrm{Aut}(\Delta)\cong\mathrm{P}\Gamma\mathrm{U}(V)$.
\end{thm}

\begin{proof}[Sketch of Proof]  Since the polar geometry $(V,\beta)$ is embeddable in a projective geometry and $\dim V\geq 4$, the Fundamental Theorem of Projective Geometry applies to ensure that every automorphism of the polar space is induced by a semilinear transformation of $V$.  It follows that the automorphism group of the polar geometry is isomorphic to $\mathrm{P}\Gamma\mathrm{U}(V)$.  Finally that every automorphism of the building arises from an automorphism of the geometry is shown on Page 264 of \cite{Pasini1994}.
\end{proof}

\begin{note}  The main interest of Theorem \ref{thm.autofbuilding} is that there is a surjective homomorphism $\Gamma\mathrm{U}(V)\twoheadrightarrow\mathrm{Aut}(\Delta)$ so that in the proof of Lemma \ref{lem.induceflip} we can argue that any flip of $\Delta$ is induced by some transformation in $\Gamma\mathrm{U}(V)$.
\end{note}

\begin{lem}\label{lem.squarescalar}  Let $\varphi$ be a flip of $\Delta$.  Then $\varphi$ is induced by a similitude $f$ of $V$ which satisfies $f^{2}=\lambda \mathrm{id}$ on $V$ for some scalar $\lambda$.
\end{lem}

\begin{proof}  Recall first that the longest word $w_{0}$ of the Weyl group of type $C_{n}$ is central.  Thus a flip $\varphi$ in fact satisfies $\delta(u,v)=\delta(\varphi(u),\varphi(v))$ and so is an automorphism of $\Delta$.  It follows from Theorem \ref{thm.autofbuilding} that $\varphi$ is induced by some semilinear map $f\in\Gamma\mathrm{U}(V)$.

Since $\varphi^{2}=\mathrm{id}$ on $\Delta$ we see that $f^{2}$ is in the kernel of the action of $\Gamma\mathrm{U}(V)$ on $\Delta$, which is $Z(V)\cap\Gamma\mathrm{U}(V)$, the group of scalar transformations that also lie in $\Gamma\mathrm{U}(V)$.  Thus $f^{2}=\lambda\mathrm{id}$ on $V$ for some $\lambda\in\mathbb{F}$.  
\end{proof}

Recall that the norm $N_{\sigma}:\mathbb{F}\to\mathbb{F}_{q}$ defined by $N_{\sigma}(x)=x\sigma(x)$ is surjective since $\mathbb{F}$ is finite.

\begin{lem}\label{lem.rightaut}  Let $\varphi$ be induced by a similitude $f$ of $V$.  Then either $f$ is linear or $f$ is $\sigma$-semilinear.  Moreover if $\beta(f(u),f(v))=a\tau(\beta(u,v))$ and $f^{2}=\lambda\mathrm{id}$ then $N_{\sigma}(\lambda)=a^{2}$.
\end{lem}

\begin{proof}  Suppose $f$ is $\tau$-semilinear for $\tau\in\mathrm{Aut}(\mathbb{F})$.  Let $\eta\in\mathbb{F}$ and let $u\in V$ with $u\neq 0$.  Since $f^{2}=\lambda\mathrm{id}$ on $V$ it follows that $f^{2}(\eta u)=\lambda\eta u$.  But we can calculate directly that 
$$
f^{2}(\eta u)=f(\tau(\eta)f(u))=\tau^{2}(\eta)f^{2}(u)=\tau^{2}(\eta)\lambda u.
$$
Thus $\tau^{2}(\eta)=\eta$ and so $\tau^{2}$ is the identity of $\mathrm{Aut}(\mathbb{F})$.  Since $\mathrm{Aut}(\mathbb{F})$ contains a unique involution it follows that either $\tau= \mathrm{id}$ and $f$ is linear, or $\tau=\sigma$ and $f$ is $\sigma$-semilinear.

In order to prove the second part of the theorem, notice that 
$$
N_{\sigma}(\lambda)\beta(u,v)=\beta(f^{2}(u),f^{2}(v))=a\tau(\beta(f(u),f(v))).
$$
Since $\beta(f(u),f(v))=a\tau\beta(u,v)$ it follows that 
$$
a\tau(\beta(f(u),f(v)))=a\tau(a)\tau^{2}(\beta(u,v))=a^{2}\beta(u,v).
$$
This string of equalities relies on the fact that either $\tau=\mathrm{id}$ or $\tau=\sigma$, and in either case $\tau(a)=a$.

Putting these two strings of equalities together we see that for all $u$, $v\in V$,
$$
N_{\sigma}(\lambda)\beta(u,v)=a^{2}\beta(u,v)
$$
and so since $\beta$ is non-degenerate, $N_{\sigma}(\lambda)=a^{2}$.
\end{proof}

\begin{lem}\label{lem.induceflip}  Let $\varphi$ be a flip of $\Delta$.  Then one of the following holds:
\begin{enumerate}
\item[(i)]  $\varphi$ is induced by a linear isometry $f\in\mathrm{U}(V)$ satisfying $f^{2}=\mathrm{id}$ on $V$; or
\item[(ii)]  $\varphi$ is induced by a linear anti-isometry $f$ of $V$ satisfying $f^{2}=\mathrm{id}$ on $V$; or
\item[(iii)]  $\varphi$ is induced by a $\sigma$-semilinear isometry $f\in\Gamma\mathrm{U}(V)$ so that $f^{2}=\mathrm{id}$ on $V$.
\end{enumerate}
\end{lem}

\begin{proof}  By Lemma \ref{lem.squarescalar} $\varphi$ is induced by a similitude $f$ of $V$ with $f^{2}=\lambda\mathrm{id}$ on $V$ for some scalar $\lambda\in\mathbb{F}$.

Since $\varphi$ maps some chamber of $\Delta$ to an opposite, there is an apartment
$$
\Sigma=\Sigma(e_{i},f_{i}|i=1,\ldots,n)
$$
in which $\varphi$ sends the chamber $C=(C_{i})_{i=1}^{n}$ defined by $C_{i}=\langle e_{1},\ldots,e_{i}\rangle$ to its opposite in $\Sigma$, the chamber $D=(D_{i})_{i=1}^{n}$ defined by $D_{i}=\langle f_{1},\ldots,f_{i}\rangle$.

Since $C$ and $D$ are opposite, they lie in a unique apartment.  It follows that $\varphi$ preserves the apartment $\Sigma$.  In particular, since for each $i=1,\ldots,n$ we have $\langle e_{i}\rangle=C_{i}\cap D_{i-1}^{\perp}$ and $\langle f_{i}\rangle =D_{i}\cap C_{i-1}^{\perp}$ we see that $\varphi$ sends each 1-object to its opposite in $\Sigma$ and so for each $i=1,\ldots,n$ there exist scalars $\lambda_{i},\mu_{i}\in\mathbb{F}$ so that
\begin{eqnarray*}
f(e_{i})&=&\lambda_{i}f_{i}\\
f(f_{i})&=&\mu_{i}e_{i}.
\end{eqnarray*}

\begin{enumerate}
\item[(a)]  Suppose $f$ is linear and for all $u$, $v\in V$, $\beta(f(u),f(v))=a\beta(u,v)$.  Since $a\in\mathbb{F}_{q}$ there exists $\mu\in\mathbb{F}$ such that $N_{\sigma}(\mu)=a^{-1}$.  Replace $f$ by $\mu f$ and we see that for all $u$, $v\in V$,
$$
\beta((\mu f)(u),(\mu f)(v))=N_{\sigma}(\mu)\beta(f(u),f(v))=a^{-1}a\beta(u,v)=\beta(u,v).
$$
Thus $\mu f$ is an isometry which also induces $\varphi$.  

Suppose now that we have chosen an isometry $f$ which induces $\varphi$, and $f^{2}=\lambda\mathrm{id}$.  If $\lambda=1$ the the conclusion of (i) is satisfied and we're done.  So assume $\lambda\neq 1$.  Notice that we have the following equalities:
\begin{eqnarray}
\beta(u,v)=\beta(f(u),f(v))&=&\beta(f^{2}(u),f^{2}(v))=N_{\sigma}(\lambda)\beta(u,v)\\
\sigma(\lambda_{1})=\beta(e_{1},f(e_{1}))&=&\beta(f(e_{1}),f^{2}(e_{1}))=\beta(\lambda_{1}f_{1},\lambda e_{1})=\lambda_{1}\sigma(\lambda).
\end{eqnarray}
It follows from (1) that $N_{\sigma}(\lambda)=1$, and from (2) that $\lambda$ is a square in $\mathbb{F}$.  Choose $\eta\in\mathbb{F}$ so that $\eta^{2}=\lambda^{-1}$.  Since $N_{\sigma}$ is multiplicative, it follows that $N_{\sigma}(\eta)^{2}=N_{\sigma}(\eta^{2})=N_{\sigma}(\lambda)=1$ and so $N_{\sigma}(\eta)\in\{\pm 1\}$.

Let $g=\eta f$.  Then $g^{2}=\mathrm{id}$ on $V$, but we have paid a price.  We now have that
$$
\beta(g(u),g(v))=\beta(\eta f(u),\eta f(v))=N_{\sigma}(\eta)\beta(u,v).
$$
Thus either $g$ is an isometry of $(V,\beta)$ or $g$ is an anti-isometry of $(V,\beta)$.  If $g$ is an isometry the conclusion of (i) is satisfied, and if $g$ is an anti-isometry the conclusion of (ii) is satisfied.

\item[(b)]  Suppose now that $f$ is semilinear but not linear.  Then by Lemma \ref{lem.rightaut} $f$ is $\sigma$-semilinear.  We now show that we can replace $f$ by a scalar multiple $\mu f$ which still induces $\varphi$ so that $(\mu f)^{2}=\mathrm{id}$ on $V$.  Namely we have
\begin{eqnarray*}
\lambda e_{i}&=& f^{2}(e_{i})=\mu_{i}\sigma(\lambda_{i})e_{i} \\
\lambda f_{i}&=& f^{2}(f_{i})=\lambda_{i}\sigma(\mu_{i})f_{i}
\end{eqnarray*}
and so $\lambda=\mu_{i}\sigma(\lambda_{i})=\lambda_{i}\sigma(\mu_{i})$.  Hence $\lambda$ lies in $\mathbb{F}^{\sigma}=\mathbb{F}_{q}$, the fixed field of $\sigma$.  Since $\mathbb{F}$ is finite the norm map $N_{\sigma}:\mathbb{F}\to\mathbb{F}_{q}$ is surjective.  Thus there exists $\mu\in\mathbb{F}$ so that $N_{\sigma}(\mu)=\lambda^{-1}$.  Replacing $f$ by $\mu f$ does not affect $\varphi$, and so we may do this and assume $\lambda=1$.

In order to check that $f$ can be taken to be an isometry, by Lemma \ref{lem.squarescalar}, since $\lambda=1$ also $a^{2}=1$.  Hence $a\in\{\pm 1\}$.  The following calculation shows that $a^{-1}=1$ and so $a=1$ and we are in the situation of (iii):
$$
1=\beta(e_{1},f_{1})=a^{-1}\sigma^{-1}(\beta(f(e_{1}),f(f_{1})))=a^{-1}\sigma^{-1}(\beta(\lambda_{1}f_{1},\mu_{1}e_{1}))=a^{-1}\lambda=a^{-1}.\qedhere
$$
\end{enumerate}
\end{proof}

It is easy to see that the three cases in Lemma \ref{lem.induceflip} are mutually exclusive.

\begin{note}  It follows immediately from Lemma \ref{lem.rightaut} that if $f$ induces a flip $\varphi$ and $f^{2}=\mathrm{id}$ on $V$ then either $f$ is an isometry or $f$ is an anti-isometry.  What is interesting about Lemma \ref{lem.induceflip} is that if $f$ is linear we have to consider both the isometry and anti-isometry possibilities, whereas if $f$ is $\sigma$-semilinear we can assume it is an isometry.
\end{note}

\begin{defn}  Let $\varphi$ be a flip of $\Delta$.  We say $\varphi$ is \textbf{linear} if it is induced by a linear transformation of $V$.  We say $\varphi$ is \textbf{$\sigma$-semilinear} if it is induced by a $\sigma$-semilinear transformation of $V$.
\end{defn}

\begin{note}  From now on we identify $\varphi$ with a transformation of $V$ that induces $\varphi$ and satisfies the appropriate conclusion of Lemma \ref{lem.induceflip}.
\end{note}

We now define a new form on $V$ that will be important in the study of geometries induced by $\varphi$.

\begin{defn}  Given $f\in\Gamma\mathrm{U}(V)$, define $\beta_{f}(u,v)=\beta(u,f(v))$.
\end{defn}

\begin{lem}\label{lem.newform}  Let $f\in\Gamma\mathrm{U}(V)$ be $\tau$-semilinear, and assume $f^{2}=\mathrm{id}$.  Then $\beta_{f}$ is a non-degenerate, reflexive, $\sigma\tau$-sesquilinear form.  In particular,
\begin{enumerate}
\item[(i)]  if $\sigma=\tau$, then $\beta_{f}$ is a non-degenerate bilinear form, and
\item[(ii)]  if $f$ is linear, then $\beta_{f}$ is a non-degenerate $\sigma$-sesquilinear form.
\end{enumerate}
\end{lem}

\begin{proof}  That $\beta_{f}$ is non-degenerate follows since $f$ is bijective.  Left homogeneity follows since $\beta$ is left homogeneous and $f$ acts in the second argument.  Reflexivity and both (i) and (ii) follow from direct calculations.
\end{proof}

\begin{note}  If $f$ induces a flip $\varphi$ as in Lemma \ref{lem.induceflip} then the choice $f$ determines $\beta_{\varphi}$ up to multiplication by $\pm 1$ if $f$ is linear, and up to multiplication by a scalar of norm 1 if $f$ is $\sigma$-semilinear.  It follows that choosing different representatives for $\varphi$ does not affect the results of Lemmas \ref{lem.sigmaform} and \ref{lem.linflipisoform}.
\end{note}

\begin{lem}\label{lem.sigmaform}  Let $\varphi$ be a $\sigma$-semilinear flip of $\Delta$.  Then $\beta_{\varphi}$ is a non-degenerate, reflexive, symmetric, bilinear form.
\end{lem}

\begin{proof}  All except the fact that $\beta_{\varphi}$ is symmetric follows from Lemma \ref{lem.newform}.  That $\beta_{\varphi}$ is symmetric follows from an easy calculation.
\end{proof}

\begin{lem}\label{lem.linflipisoform}  Let $\varphi$ be a linear flip of $\Delta$.  
\begin{enumerate}
\item[(i)]  If $\varphi$ is a linear isometry then $\beta_{\varphi}$ is a non-degenerate, reflexive, $\sigma$-hermitian form.
\item[(ii)]  If $\varphi$ is a linear anti-isometry then $\beta_{\varphi}$ is a non-degenerate, reflexive, $\sigma$-antihermitian form.
\end{enumerate}
\end{lem}

\begin{proof}  All that remains is to show that in (i) the form is hermitian and in (ii) the form is antihermitian.  Both follow from easy calculations.
\end{proof}

\begin{defn}  Given a flip $\varphi$, set $Q_{\varphi}(v)=\frac{1}{2}\beta_{\varphi}(v,v)$.
\end{defn}

Notice that $Q_{\varphi}$ is the pseudo-quadratic form that polarizes to $\beta_{\varphi}$.

\subsection{The Chamber System Induced by a Flip}

We now define a chamber system left invariant by a flip.  We shall use this to classify $\sigma$-semilinear flips, but we will also be interested in these for their automorphism groups.

\begin{defn}  By $\Delta^{\varphi}$ we denote the collection of chambers of $\Delta$ sent to an opposite chamber by $\varphi$.
\end{defn}

\begin{defn}  Recall that a pair of vectors $u$, $v\in V$ are \textbf{$\beta$-orthogonal} if $\beta(u,v)=0$, this is denoted $u\perp v$.  The vectors are \textbf{$\beta_{\varphi}$-orthogonal} if $\beta_{\varphi}(u,v)=0$, this is denoted $u\perp_{\varphi} v$.  The vectors are \textbf{biorthogonal} if $\beta(u,v)=\beta_{\varphi}(u,v)=0$, this is denoted $u\dperp v$.  If $U$ is a subspace of $V$ we use $U^{\perp}$, $U^{\perp_{\varphi}}$, and $U^{\dperp}$ to refer to the $\beta$-orthogonal complement, $\beta_{\varphi}$-orthogonal complement, and biorthogonal complement respectively.
\end{defn}

Recall that a pair of $\beta$ isotropic vectors $u$, $v$ is called a \textbf{hyperbolic pair} if $\beta(u,v)=1$.  We define a \textbf{pre-hyperbolic pair} to be a pair of $\beta$ isotropic vectors $u$, $v$ with $\beta(u,v)\neq 0$.  This is not standard, but there are instances where the distinction will be important.

\begin{thm}\label{thm.idchambers}  Let $\varphi$ be a flip of the unitary building $\Delta$.
\begin{enumerate}
\item[(1)]  A chamber $C=(C_{i})_{i=1}^{n}$ of $\Delta$ lies in $\Delta^{\varphi}$ if and only if $C_{i}$ is non-degenerate with respect to $\beta_{\varphi}$ for all $i=1,\ldots,n$.

\item[(2)]  If $\{e_{i},f_{i}\}_{i=1}^{n}$ is a $\beta$ pre-hyperbolic basis for $V$ with $\varphi(e_{i})=f_{i}$, then the chambers $(C_{i})_{i=1}^{n}$ and $(D_{i})_{i=1}^{n}$ defined by $C_{i}=\langle e_{1},\ldots,e_{i}\rangle$ and $D_{i}=\langle f_{1},\ldots,f_{i}\rangle$ are opposite in $\Delta$ and so lie in $\Delta^{\varphi}$.  Conversely if $C=(C_{i})_{i=1}^{n}$ is a chamber of $\Delta^{\varphi}$, then there is a $\beta$ pre-hyperbolic basis $\{e_{i},f_{i}\}$ for $V$ so that $C_{i}=\langle e_{1},\ldots,e_{i}\rangle$, $\varphi(e_{i})=f_{i}$ for all $i=1,\ldots,n$, and the chamber $D=(D_{i})_{i=1}^{n}$ defined by $D_{i}=\langle f_{1},\ldots,f_{i}\rangle$ lies in $\Delta^{\varphi}$ and is opposite to $C$.
\end{enumerate}
\end{thm}

\begin{proof}
\begin{enumerate}
\item[(1)]  By assumption we may view $\varphi$ as acting on the vector space $V$, and have $\varphi^{2}=\mathrm{id}$ on $V$.  Suppose $C=(C_{i})_{i=1}^{n}$ is a chamber of $\Delta^{\varphi}$.  Then since $C$ is also a chamber of $\Delta$, each $C_{i}$ is $\beta$ isotropic.

    Recall from Theorem \ref{thm.oppositechambers} that $C$ is opposite to $\varphi(C)$ in $\Delta$ if and only if for each $i$,
    $$
    \varphi(C_{i})\cap C_{i}^{\perp}=\{0\}.
    $$
Notice that 
$$
\varphi(\mathrm{Rad}_{\beta_{\varphi}}(C_{i}))=\varphi(C_{i}\cap C_{i}^{\perp_{\varphi}})=\varphi(C_{i})\cap C_{i}^{\perp}
$$
where the last equality is justified since $\varphi$ is a bijective transformation of $V$.  Thus the $\beta_{\varphi}$ radical of $C_{i}$ is $\{0\}$ if and only if $\varphi(C_{i})\cap C_{i}^{\perp}=\{0\}$.

\item[(2)]  The first part follows from (1) by noting that if $\{e_{i},f_{i}\}$ is a $\beta$ pre-hyperbolic basis for $V$ with $f_{i}=\varphi(e_{i})$ for all $i$, then for each $i$, $\{e_{1},\ldots,e_{i}\}$ and $\{f_{1},\ldots,f_{i}\}$ form $\beta_{\varphi}$ orthogonal bases for $C_{i}$ and $D_{i}$ respectively, and satisfy the hypotheses of (1).

    Conversely suppose $C=(C_{i})_{i=1}^{n}$ is a chamber of $\Delta^{\varphi}$.  Choose $e_{1},\ldots,e_{n}$ as follows. Pick $e_{1}\in C_{1}-\{0\}$.  Then pick $e_{i}\in C_{i}\cap \varphi(C_{i-1})^{\perp}$.  The vectors $e_{1},\ldots,e_{n}$ are pairwise biorthogonal.  Moreover none can be $\beta_{\varphi}$ isotropic as this would contradict the $\beta_{\varphi}$ non-degeneracy of $C_{i}$.  Finally, define $f_{i}=\varphi(e_{i})$ for $i=1,\ldots,n$.  Then $\{e_{i},f_{i}\}_{i=1}^{n}$ gives the desired basis.\qedhere
\end{enumerate}
\end{proof}

\begin{lem}\label{lem.converse}  If $f\in\Gamma\mathrm{U}(V)$ satisfies any of (i)-(iv) in the statement of Main Theorem 1 then $f$ induces a flip of $\Delta$.
\end{lem}

\begin{proof}  Since $f\in\Gamma\mathrm{U}(V)$, $f$ induces an automorphism of $\Delta$ and by assumption $f$ has order 2.  It therefore suffices to show that $f$ maps some chamber of $\Delta$ to an opposite chamber.  Let $\{e_{i},f_{i}\}_{i=1}^{n}$ be a hyperbolic basis for $V$ as in the hypotheses of Main Theorem 1.  Let $C_{i}=\langle e_{1},\ldots,e_{i}\rangle$ for $i=1,\ldots,n$.  Then $\varphi(C_{i})=\langle f_{1},\ldots,f_{i}\rangle$, and clearly $C_{i}^{\perp}\cap\varphi(C_{i})=\{0\}$ for all $i$.  Hence the chamber $C=(C_{i})_{i=1}^{n}$ is sent to an opposite chamber by $f$, and so $f$ induces a flip of $\Delta$.
\end{proof}

\subsection{The Geometry Induced by a Flip}\label{sec.geom}

We now define the geometry corresponding to the chamber system induced by a flip.

\begin{defn}\label{def.gamma}  Let $\Gamma(n,q)$ denote the set of all $\beta$-isotropic and $\beta_{\varphi}$ non-degenerate subspaces $U$ of $V$.  Let $I=\{1,\ldots,n\}$ and define $\tau:\Gamma(n,q)\to I$ by $\tau(U)=\dim(U)$.  Finally, define a relation $\sim$ on $\Gamma(n,q)$ by $U\sim W$ if $U\subseteq W$ or $W\subseteq U$.  $(\Gamma(n,q),\tau,I,\sim)$ is the \textbf{geometry induced by $\varphi$}.
\end{defn}

\begin{lem}  Let $U$ be a $\beta$-isotropic subspace of $V$.  Then $U\in\Gamma(n,q)$ if and only if $U^{\perp}\cap\varphi(U)=\{0\}$.
\end{lem}

\begin{proof}  $U$ is $\beta_{\varphi}$ non-degenerate if and only if $U^{\perp}\cap\varphi(U)=\{0\}$.
\end{proof}

\begin{note}  It is clear that $(\Gamma(n,q),I,\tau,\sim)$ is a pregeometry, since $\Gamma(n,q)$ is a subset of the set of objects of the full projective geometry of $V$, $\mathcal{P}(V)$, and we have inherited the type and incidence structure from $\mathcal{P}(V)$.  We will prove in Theorem \ref{thm.biggeometry} that $\Gamma(n,q)$ is in fact a geometry.  In order to achieve this goal we will have to study the properties of the vector space endowed with both forms $\beta$ and $\beta_{\varphi}$ in more detail.
\end{note}

\begin{defn}  A \textbf{point} of $\Gamma(n,q)$  is an object of $\Gamma(n,q)$ of type 1.  A \textbf{line} of $\Gamma(n,q)$ is an object of $\Gamma(n,q)$ of type 2.
\end{defn}

\begin{note}  From now on, we have chosen to identify a point of the geometry, which is really a 1-dimensional subspace of $V$, with a non-zero vector in that subspace.
\end{note}

The proofs of Lemmata \ref{lem.multiform1} and \ref{lem.multiform2} are straightforward.

\begin{lem}\label{lem.multiform1}  Let $U$ be a subspace of $V$.  Then $U^{\perp}=\varphi(U)^{\perp_{\varphi}}$ and 
$$
\langle U,\varphi(U)\rangle^{\perp}=\langle U,\varphi(U)\rangle^{\perp_{\varphi}}=U^{\dperp}.
$$
\end{lem}

\begin{lem}\label{lem.multiform2}  Let $U$, $U'\in\Gamma(n,q)$ with $U\subset U'$.  Then
$$
\langle U,\varphi(U)\rangle^{\perp}\cap U'=\varphi(U)^{\perp}\cap U'=U^{\perp_{\varphi}}\cap U'=U^{\dperp}\cap U'.
$$
\end{lem}

\begin{lem}\label{lem.complementingamma}  Let $U$, $U'\in\Gamma(n,q)$ with $U\subset U'$.  Then
$$
W=\langle U,\varphi(U)\rangle^{\perp}\cap U'\in \Gamma(n,q).
$$
\end{lem}

\begin{proof}  Since $W\subset U'$ and $U'$ is $\beta$ isotropic, also $W$ is $\beta$ isotropic.  By Lemma \ref{lem.multiform2} $W=U^{\perp_{\varphi}}\cap U'$ and so $W$ is a $\beta_{\varphi}$ orthogonal complement to $U$ in $U'$.  Since both $U$ and $U'$ are $\beta_{\varphi}$ non-degenerate it follows that $W$ is $\beta_{\varphi}$ non-degenerate.  
\end{proof}

\begin{lem}\label{lem.gammapointexist}  Let $U\in\Gamma(n,q)$.  Then $U$ contains a point of $\Gamma(n,q)$.
\end{lem}

\begin{proof}  If $\dim U=1$ then $U$ is a point.  Suppose $\dim U>1$.  Since $U\in\Gamma(n,q)$ it is $\beta_{\varphi}$ non-degenerate and so there exists $u$, $v\in U$ so that $\beta_{\varphi}(u,v)\neq 0$.  If either $Q_{\varphi}(u)\neq 0$ or $Q_{\varphi}(v)\neq 0$ then $u$ or $v$ is a point respectively.  Otherwise it is straightforward to check that there exists $\lambda\in\mathbb{F}$ so that $u+\lambda v$ is a point of $\Gamma(n,q)$.
\end{proof}

\subsection{Further Properties of $\beta$ and $\beta_{\varphi}$}

In this section we have collected some results concerning the relationship between $\beta$ and $\beta_{\varphi}$.  These results hold for both linear and $\sigma$-semilinear flips and will be used in showing that $\Gamma(n,q)$ is a geometry.

Recall that $\varphi$ denotes both a flip and a semilinear transformation of $V$ that induces the flip and satisfies the appropriate conclusion of Lemma \ref{lem.induceflip}.

\begin{lem}\label{lem.invradical}  Let $U$ be a $\varphi$-invariant subspace of $V$.  Then
$$
\mathrm{Rad}_{\beta}(U)=\mathrm{Rad}_{\beta_{\varphi}}(U)=\varphi(\mathrm{Rad}_{\beta}(U)).
$$
\end{lem}

\begin{proof}  A vector $u$ lies in the $\beta$ radical of $U$ if and only if $\beta(u,v)=0$ for all $v\in U$.  Since $U$ is $\varphi$-invariant, $U=\varphi(U)$ and so this is also equivalent to requiring that $\beta_{\varphi}(u,v)=0$ for all $v\in U$.  Thus a vector lies in the $\beta$ radical of $U$ if and only if it lies in the $\beta_{\varphi}$ radical of $U$.
\end{proof}

\begin{note}  From now on, when referring to the radical of a $\varphi$-invariant subspace we need not specify to which form we are referring.
\end{note}

\begin{lem}\label{lem.maschke}  Let $U$ be a $\varphi$-invariant subspace of $V$, and let $R$ be a $\varphi$-invariant subspace of $U$.  Then $R$ has a $\varphi$-invariant complement in $U$.
\end{lem}

\begin{proof}  This is a special case of Maschke's Theorem, see for example Theorem 1.9 of \cite{Isaacs2006}.
\end{proof}

Combining Lemmata \ref{lem.invradical} and \ref{lem.maschke} we immediately find the following.

\begin{cor}  Let $U$ be a $\varphi$-invariant subspace of $V$ and let $R$ be its radical.  Then $R$ has a $\varphi$-invariant complement in $U$.
\end{cor}

\begin{lem}\label{lem.residuenondeg}  Let $W\in\Gamma(n,q)$ with $\dim W=k$.  Then $W\cap\varphi(W)=\{0\}$.  Hence $W'=\langle W,\varphi(W)\rangle$ is $2k$-dimensional, $\varphi$-invariant, and non-degenerate.
\end{lem}

\begin{proof}  Since $W$ is $\beta$ isotropic, $W\subseteq W^{\perp}$.  Since $W$ is $\beta_{\varphi}$ non-degenerate it follows that $$
W^{\perp}\cap \varphi(W)=\{0\}
$$
and so also $W\cap \varphi(W)=\{0\}$.  This shows that $\dim\langle W,\varphi(W)\rangle=2k$.  

That $W'$ is $\varphi$-invariant is clear.  To show that $W'$ is non-degenerate, notice first that by Lemma \ref{lem.invradical}, $\mathrm{Rad}_{\beta}(W')=\varphi(\mathrm{Rad}_{\beta}(W'))$.  Furthermore, $\mathrm{Rad}_{\beta}(W')\subseteq W^{\perp}\cap W'=W$.  But also $\varphi(\mathrm{Rad}_{\beta}(W'))\subseteq \varphi(W^{\perp}\cap W')=\varphi(W\cap W')=\varphi(W)\cap\varphi(W')=\varphi(W)$.  It follows that $\mathrm{Rad}_{\beta}(W')\subseteq W\cap \varphi(W)=\{0\}$ and so $W'$ is non-degenerate.
\end{proof}

\begin{cor}\label{cor.basebiorthopoints}  If $W\in\Gamma(n,q)$ with $\dim W=k$ then there is a basis $\{w_{i}\}_{i=1}^{k}$ for $W$ of biorthogonal points.
\end{cor}

\begin{proof}  We induct on $k$.  If $k=1$ the result is trivial.  If $k>1$ then by Lemma \ref{lem.gammapointexist} $W$ contains a point $w_{1}$ of $\Gamma(n,q)$.  Let $W'=\langle w_{1}\rangle^{\dperp}\cap W$.  It follows from Lemma \ref{lem.multiform2} that $W'=\langle w_{1}\rangle^{\perp_{\varphi}}\cap W$, and so $W'$ has codimension 1 in $W$.  By Lemma \ref{lem.complementingamma} $W'\in\Gamma(n,q)$ and so by the inductive hypothesis there exists a collection of biorthogonal points $\{w_{2},\ldots,w_{k}\}$ that is a basis for $W'$.  Since $W'\subset \langle w_{1},\varphi(w_{1})\rangle^{\perp}$ it follows that $\{w_{1},\ldots,w_{k}\}$ forms a basis of biorthogonal points for $W$.
\end{proof}

\subsection{$\Gamma(n,q)$ is a geometry}

In this section we prove that $\Gamma(n,q)$ is a geometry.  Throughout this section $\varphi$ denotes both a flip, and a semilinear transformation of $V$ that induces the flip and satisfies the appropriate conclusion of Lemma \ref{lem.induceflip}.  Unless otherwise stated these results hold for both linear and $\sigma$-semilinear flips.

\begin{lem}\label{lem.notscalar}  Let $U$ be a subspace of $V$ with $\dim U>n$.  Then $\varphi$ does not act as a scalar on $U$.
\end{lem}

\begin{proof}   Let $M$ be an $n$-dimensional $\beta$ isotropic $\beta_{\varphi}$ non-degenerate subspace of $V$.  Then by Lemma \ref{lem.residuenondeg}, $M\cap\varphi(M)=\{0\}$.  If $U$ is a subspace of dimension greater than $n$ and $\varphi$ acts as a scalar on $U$, then $\varphi$ acts as a scalar on $M\cap U\neq\{0\}$.  Thus there is a non-zero vector $v\in M\cap U$ with $\varphi(v)=\mu v$ for some non-zero $\mu\in\mathbb{F}$.  But then $v\in M\cap \varphi(M)=\{0\}$, a contradiction.
\end{proof}

\begin{lem}\label{lem.linflipscalar}  Suppose $\varphi$ is a linear flip, and let $X$ be $2k$-dimensional, $\varphi$ invariant and non-degenerate subspace of $V$.  Then one of the following three holds:
\begin{enumerate}
\item[(i)]  $X$ contains a point of $\Gamma(n,q)$;
\item[(ii)]  $\varphi(x)=x$ for all $x\in X$;
\item[(iii)]  $\varphi(x)=-x$ for all $x\in X$.
\end{enumerate}
\end{lem}

\begin{proof}  Suppose that $X$ does not contain any points of $\Gamma(n,q)$.  We will show that either (ii) or (iii) holds. Since $X$ is $\beta$ non-degenerate and even dimensional we can write it as an orthogonal direct sum of $\beta$ hyperbolic lines,
$$
X=\perp_{i=1}^{m}\langle a_{i},b_{i}\rangle
$$
where each $(a_{i},b_{i})$ is a hyperbolic pair.

We proceed now in a series of steps to show that $\varphi$ acts on $X$ as either $\mathrm{id}_{X}$ or $-\mathrm{id}_{X}$.
\begin{enumerate}
\item[]\textbf{Step 1}:  If $u$, $v\in X$ are $\beta$-isotropic then $\beta(u,v)=0$ if and only if $\beta_{\varphi}(u,v)=0$.
\begin{proof}  Notice first that since $X$ contains no points of $\Gamma(n,q)$, $Q_{\varphi}(u)=Q_{\varphi}(v)=0$.

Suppose $\beta(u,v)=0$ but $\beta_{\varphi}(u,v)\neq 0$.  If $\varphi$ is an isometry and $\lambda$ is chosen so that $\mathrm{Tr}_{\sigma}(\sigma(\lambda)\beta_{\varphi}(u,v))\neq 0$ then $u+\lambda v$ is a point of $\Gamma(n,q)$.  If $\varphi$ is an anti-isometry and $\lambda$ is chosen so that $\sigma(\lambda)\beta_{\varphi}(u,v)-\lambda\sigma(\beta_{\varphi}(u,v))\neq 0$ then $u+\lambda v$ is a point of $\Gamma(n,q)$.  In either case, such $\lambda$ exist and so since by hypothesis $X$ contains no points of $\Gamma(n,q)$ we conclude that if $\beta(u,v)=0$ then $\beta_{\varphi}(u,v)=0$.

Conversely if $\beta_{\varphi}(u,v)=0$ but $\beta(u,v)\neq 0$ then $\beta(u,\varphi(v))=0$ while $\beta_{\varphi}(u,\varphi(v))\neq 0$, which we have already shown cannot happen.

Thus $\beta(u,v)=0$ if and only if $\beta_{\varphi}(u,v)=0$ for all $\beta$-isotropic $u$, $v\in X$.
\end{proof}

\item[]\textbf{Step 2}:  For all $i=1,\ldots m$, $\varphi(a_{i})\in\langle a_{i}\rangle$ and $\varphi(b_{i})\in\langle b_{i}\rangle$.
\begin{proof}  We perform the calculation only for $a_{1}$, the others are similar.  Suppose
$$
\varphi(a_{1})=\sum_{i=1}^{m}(x_{i}a_{i}+y_{i}b_{i})
$$
for some scalars $x_{i}$, $y_{i}\in\mathbb{F}$, $i=1,\ldots,m$.

Since $\beta(b_{i},a_{1})=0$ for all $i\neq 1$, also $\beta_{\varphi}(b_{i},a_{1})=0$ for all $i\neq 1$.  But we can calculate that $\beta_{\varphi}(b_{i},a_{1})=\sigma(x_{i})$, and so $x_{i}=0$ if $i\neq 1$.

Similarly for all $i\neq 1$, $\beta(a_{i},a_{1})=0$ and so also $\beta_{\varphi}(a_{i},a_{1})=0$, but $\beta_{\varphi}(a_{i},a_{1})=\sigma(y_{i})$ and so $y_{i}=0$.

Hence $\varphi(a_{1})=x_{1}a_{1}$.
\end{proof}

\item[]\textbf{Step 3}:  For all $i$, $\varphi(a_{i})=a_{i}$ or $\varphi(a_{i})=-a_{i}$.  Similarly $\varphi(b_{i})=b_{i}$ or $\varphi(b_{i})=-b_{i}$.
    \begin{proof}  We prove the result for $a_{i}$, the result for $b_{i}$ is proved similarly.
    Since $\varphi^{2}=\mathrm{id}$ on $V$, $\varphi^{2}(a_{i})=x_{i}^{2}a_{i}=a_{i}$ and so $x_{i}^{2}=1$.  Hence $x_{i}\in\{\pm 1\}$.
    \end{proof}

\item[]\textbf{Step 4}:  $\varphi(a_{i})=-a_{i}$ if and only if $\varphi(b_{i})=-b_{i}$.
\begin{proof}  Assume first that $\varphi$ is an isometry of $(V,\beta)$ and that $\varphi(a_{i})=-a_{i}$.  Then
$$
1=\beta(a_{i},b_{i})=\beta(\varphi(a_{i}),\varphi(b_{i}))=-\beta(a_{i},\varphi(b_{i}))
$$
which forces $\varphi(b_{i})=-b_{i}$.  Similarly if $\varphi(b_{i})=-b_{i}$ then $\varphi(a_{i})=-a_{i}$.

Assume next that $\varphi$ is an anti-isometry of $(V,\beta)$ and that $\varphi(a_{i})=-a_{i}$ while $\varphi(b_{i})=b_{i}$.  Consider the vector $x=a_{i}+\lambda b_{i}$ where $\lambda$ is any non-zero element of trace 0 in $\mathbb{F}$.  An easy calculation shows that $\beta(x,x)=0$ while $\beta_{\varphi}(x,x)=-2\lambda\neq 0$.  Thus $x$ is a point of $\Gamma(n,q)$ which lies in $X$, contradicting the assumption that $X$ contains no points of $\Gamma(n,q)$.
\end{proof}

\item[]\textbf{Step 5}:  If $\varphi(a_{1})=a_{1}$ then $\varphi(a_{i})=a_{i}$ for all $i$ and if $\varphi(a_{1})=-a_{1}$ then $\varphi(a_{i})=-a_{i}$ for all $i$.
    \begin{proof}  Suppose that $\varphi(a_{1})=a_{1}$ but $\varphi(a_{i})=-a_{i}$ for some $i$.  Then also $\varphi(b_{1})=b_{1}$ and $\varphi(b_{i})=-b_{i}$.  Let $x=a_{1}+b_{1}+a_{i}-b_{i}$.  Then two easy calculations show that $x$ is a point of $\Gamma(n,q)$.  Since by assumption $X$ contains no points of $\Gamma(n,q)$ we conclude that if $\varphi(a_{1})=a_{1}$ then $\varphi(a_{i})=a_{i}$ for all $i$.  Similarly if $\varphi(a_{1})=-a_{1}$ then $\varphi(a_{i})=-a_{i}$ for all $i$.
    \end{proof}
\end{enumerate}
Thus for all $x\in X$, either $\varphi(x)=x$ or $\varphi(x)=-x$.
\end{proof}

The situation is even better for a $\sigma$-semilinear flip:

\begin{lem}\label{lem.evendimpoint}  Suppose $\varphi$ is $\sigma$-semilinear and let $U$ be a $2k$-dimensional ($k\geq 1$) subspace of $V$ that is $\beta$ non-degenerate.  Then either $U$ is $\beta_{\varphi}$ totally singular or $U$ contains a point of $\Gamma(n,q)$.
\end{lem}

\begin{proof}  Assume $U$ is not $\beta_{\varphi}$ totally singular.  Since $U$ is $\beta$ non-degenerate we can write $U=\perp_{i=1}^{k}\langle a_{i},b_{i}\rangle$ where each $(a_{i},b_{i})$ is a hyperbolic pair.  If any $a_{i}$ or $b_{j}$ is a point of $\Gamma(n,q)$ then it is the desired point.  So we may assume that for all $i$, $Q_{\varphi}(a_{i})=Q_{\varphi}(b_{i})=0$.

Since $U$ is not $\beta_{\varphi}$ totally singular we must have one of the following.
\begin{enumerate}
\item[(i)]  There is some $i$ so that $\beta_{\varphi}(a_{i},b_{i})\neq 0$.  Then for any non-zero $\lambda$ of trace 0, $a_{i}+\lambda b_{i}$ is a point of $\Gamma(n,q)$.
\item[(ii)]  There are $i$, $j$ so that $\beta_{\varphi}(a_{i},a_{j})\neq 0$.  Then $a_{i}+a_{j}$ is a point of $\Gamma(n,q)$.
\item[(iii)]  There are $i$, $j$ so that $\beta_{\varphi}(a_{i},b_{j})\neq 0$.  Then $a_{i}+b_{j}$ is a point of $\Gamma(n,q)$.
\item[(iv)]  There are $i$, $j$ so that $\beta_{\varphi}(b_{i},b_{j})\neq 0$,  Then $b_{i}+b_{j}$ is a point of $\Gamma(n,q)$.\qedhere
\end{enumerate}
\end{proof}

\begin{cor}\label{cor.sigma}  Suppose $\varphi$ is $\sigma$-semilinear and $U$ is a $2k$-dimensional ($k\geq 1$) subspace of $V$ that is $\varphi$-invariant and non-degenerate.  Then $U$ contains a point of $\Gamma(n,q)$.
\end{cor}

\begin{thm}\label{thm.linextension}  Let $U\in\Gamma(n,q)$.  If $\dim U<n$ then the space $X=\langle U,\varphi(U)\rangle^{\perp}$ contains a point of $\Gamma(n,q)$.
\end{thm}

\begin{proof}   Notice first that since $X^{\perp}=\langle U,\varphi(U)\rangle$ is non-degenerate by Lemma \ref{lem.residuenondeg}, and $V$ is non-degenerate by hypothesis, also $X$ is non-degenerate.  

If $\varphi$ is $\sigma$-semilinear the result now follows immediately from Corollary \ref{cor.sigma}.  

Now suppose that $\varphi$ is linear.  We proceed by contradiction.  Suppose $X$ does not contain a point of $\Gamma(n,q)$.  Then by Lemma \ref{lem.linflipscalar}, $\varphi$ acts either as $\mathrm{id}_{X}$ or $-\mathrm{id}_{X}$ on $X$.  Let $k=\dim U$.  Choose a basis $\{a_{1},\ldots,a_{2(n-k)}\}$ for $X$.

Let $\{u_{1},\ldots,u_{k}\}$ be a basis of biorthogonal points for $U$.  Recall that such a basis exists by Corollary \ref{cor.basebiorthopoints}.  Then
$$
\{u_{1},\ldots,u_{k},\varphi(u_{1}),\ldots,\varphi(u_{k})\}
$$
forms a basis for $\langle U,\varphi(U)\rangle$.  We define a new basis for $\langle U,\varphi(U)\rangle$ by:
$$
\{u_{i}+\varphi(u_{i}),u_{i}-\varphi(u_{i})|i=1,\ldots,k\}.
$$
If $\varphi$ acts on $X$ as $\mathrm{id}_{X}$, define a subspace $A$ of $V$ by
$$
A=\langle u_{1}+\varphi(u_{1}),\ldots,u_{k}+\varphi(u_{k}),a_{1},\ldots,a_{2(n-k)}\rangle.
$$
Then $A$ is a $2n-k>n$ dimensional subspace of $V$ on which $\varphi$ acts as multiplication by 1, contradicting Lemma \ref{lem.notscalar}.

Thus by Lemma \ref{lem.linflipscalar}, $\varphi$ must act on $X$ as $-\mathrm{id}_{X}$.  In this case we define a subspace $B$ of $V$ by
$$
B=\langle u_{1}-\varphi(u_{1}),\ldots,u_{k}-\varphi(u_{k}),a_{1},\ldots,a_{2(n-k)}\rangle.
$$
Then $B$ is a $2n-k>n$ dimensional subspace of $V$ on which $\varphi$ acts as multiplication by $-1$, contradicting Lemma \ref{lem.notscalar}.  Hence $X$ must contain a point of $\Gamma(n,q)$.
\end{proof}

\begin{cor}\label{cor.maxobjects}  If $U$ is a maximal object of $\Gamma(n,q)$ then $\dim U=n$.
\end{cor}

\begin{proof}  We proceed by contraposition.  If $U\in\Gamma(n,q)$ with $\dim U<n$ then by Theorem \ref{thm.linextension} there is a point of $\Gamma(n,q)$, $u\in\langle U,\varphi(U)\rangle^{\perp}$.  It is easy to see that $\langle U,u\rangle\in\Gamma(n,q)$ and so $U$ is not maximal.
\end{proof}

\begin{defn}  Given an object $U\in\Gamma(n,q)$ and a subspace $X$ of $V$, define $r_{U}(X)=X\cap U^{\dperp}$.
\end{defn}

\begin{lem}\label{lem.restrictflip}  Let $U$ be an $m$-object of $\Gamma(n,q)$ with $m<n$ and let $W=U^{\dperp}$.  Then $\varphi|_{W}$ is a flip of the building of totally isotropic subspaces of $(W,\beta|_{W})$.
\end{lem}

\begin{proof}  Let $M$ be a maximal object of $\Gamma(n,q)$ containing $U$.  Let $M'=W\cap M$.  By Corollary \ref{cor.basebiorthopoints} $M'$ has a basis $\{m_{1},\ldots,m_{n-m}\}$ of biorthogonal points of $\Gamma(n,q)$.  It is easy to see that $\{m_{i},\varphi(m_{i})\}_{i=1}^{n-m}$ forms a basis for $W$.

Finally, we see that the if for $i=1,\ldots,n-m$ if we define $D_{i}=\langle m_{1},\ldots,m_{i}\rangle$ then $D=(D_{i})_{i=1}^{n-m}$ is a chamber of the building of totally isotropic subspaces of $(W,\beta|_{W})$ and by Theorem \ref{thm.oppositechambers} we see that $D$ is sent to its opposite by $\varphi|_{W}$.  Thus $\varphi_{W}$ is a flip of $(W,\beta|_{W})$.
\end{proof}

We have ignored a subtle point:  since there is more than one type of flip, which sort of flip is $\varphi|W$?  Once we finish proving Main Theorem 1 it will be easy to see that if $\varphi$ satisfies (i) or (ii) of Main Theorem 1, then so does $\varphi|_{W}$.  If $\varphi$ satisfies (iii) (resp. (iv)) of Main Theorem 1 and the determinant of the $\beta_{\varphi}$ Gram matrix of $U$ is a square in $\mathbb{F}$ then $\varphi|_{W}$ also satisfies (iii)  (resp. (iv)).  If $\varphi$ satisfies (iii) (resp. (iv)) and the determinant of the $\beta_{\varphi}$ Gram matrix of $U$ is a non-square in $\mathbb{F}$ then $\varphi|_{W}$ satisfies (iv) (resp. (iii)).

\begin{cor}\label{cor.gammaresidue}  If $u$ is a point of $\Gamma(n,q)$ then $r_{u}$ induces an isomorphism of geometries $\mathrm{res}_{\Gamma(n,q)}(u)\to\Gamma(n-1,q)$.  
\end{cor}

\begin{proof}  Notice that the objects in the residue of $u$ correspond to $\beta$ isotropic $\beta_{\varphi}$ non-degenerate subspaces of $W=\langle u,\varphi(u)\rangle^{\perp}$, and this correspondence preserves incidence.  By Lemma \ref{lem.restrictflip} $\varphi|_{W}$ is a flip of $(W,\beta|_{W})$ and it is clear from the construction that the geometry induced on $W$ by $\varphi|_{W}$ and the geometry on $W$ induced by $\varphi$ agree.  Hence $\mathrm{res}_{\Gamma(n,q)}(u)\cong\Gamma(n-1,q)$ and the isomorphism is induced by $r_{u}$.
\end{proof}

\begin{thm}\label{thm.biggeometry}  $\Gamma(n,q)$ is a geometry with type and incidence as defined in Definition \ref{def.gamma}.
\end{thm}

\begin{proof}  We induct on $n$.  If $n=1$ then the result is trivial.  Suppose $n>1$ and let $\mathcal{F}$ be a flag of $\Gamma(n,q)$.  By Lemma \ref{lem.gammapointexist} we can assume that $\mathcal{F}$ contains a point $u$ of $\Gamma(n,q)$.  By Corollary \ref{cor.gammaresidue} the residue of $u$ is isomorphic to $\Gamma(n-1,q)$, which by the inductive hypothesis is a geometry.  Thus $r_{u}(\mathcal{F})$ is a chamber in $\Gamma(n-1,q)$.  It follows easily that $\mathcal{F}$ is a chamber of $\Gamma(n,q)$.
\end{proof}

With Theorem \ref{thm.biggeometry} in hand, we can also prove the following:

\begin{lem}\label{lem.hypbasis}  Let $W$ be an object of $\Gamma(n,q)$ and let $M$ be an $n$-object of $\Gamma(n,q)$ that contains $W$.  Then any $\beta_{\varphi}$ orthogonal basis for $W$ extends to a $\beta_{\varphi}$ orthogonal basis for $M$.  Furthermore if $\{w_{i}\}_{i=1}^{n}$ is any $\beta_{\varphi}$ orthogonal basis for $M$ then $\{w_{i},\varphi(w_{i})\}_{i=1}^{n}$ forms a $\beta_{\varphi}$ orthogonal basis for $V$.
\end{lem}

\begin{proof}  Notice first that since $\Gamma(n,q)$ is a geometry, $W$ is contained in an $n$-dimensional object $M$ of $\Gamma(n,q)$.  Let $d=\dim(W)$ and let $\{w_{1},\ldots,w_{m}\}$ be a $\beta_{\varphi}$ orthogonal basis for $W$.  Then $\{w_{1},\ldots,w_{d}\}$ is a basis of biorthogonal points for $W$.  Let $\{w_{d+1},\ldots,w_{n}\}$ be a basis of biorthogonal points for $\langle W,\varphi(W)\rangle^{\perp}\cap M$, such a basis exists by combining Lemma \ref{lem.complementingamma} and Corollary \ref{cor.basebiorthopoints}.  Then $\{w_{1},\ldots,w_{n}\}$ forms a basis of biorthogonal points for $M$, which is in particular a $\beta_{\varphi}$ orthogonal basis for $M$.

That $\{w_{i},\varphi(w_{i})\}_{i=1}^{n}$ forms a $\beta_{\varphi}$ orthogonal basis for $V$ follows since $\dim M=n$ and $M\in\Gamma(n,q)$.
\end{proof}

\begin{cor}\label{cor.extendbasis}  If $W$ is a $k$-object of $\Gamma(n,q)$ with $\beta_{\varphi}$ orthogonal basis $\{w_{i}\}_{i=1}^{k}$ then there is a basis for $V$ of $\beta$ pre-hyperbolic pairs $\{e_{i},\varphi(e_{i})\}_{i=1}^{n}$ with $e_{i}=w_{i}$ for $i=1,\ldots,k$.
\end{cor}

\begin{proof}  This follows immediately from Lemma \ref{lem.hypbasis}, once one notices that if $\{e_{1},\ldots,e_{k}\}$ is a $\beta_{\varphi}$ orthogonal basis for $W$ then for each $i=1,\ldots,k$, $\beta(e_{i},\varphi(e_{i}))\neq 0$ since $\beta_{\varphi}(e_{i},e_{i})\neq 0$.
\end{proof}

\section{Linear Flips}

Throughout this section, $\varphi$ denotes a linear flip of $\Delta$.  Recall that we have identified $\varphi$ with a linear transformation of $V$ that induces $\varphi$ and satisfies the appropriate conclusion of Lemma \ref{lem.induceflip}.

\subsection{Classification of Linear Flips of the Unitary Building}\label{sec.linflips}

\begin{linflips}\label{thm.linearflips}  Let $\varphi$ be a linear flip of $\Delta$
\begin{enumerate}
\item[(i)]   If $\varphi$ is induced by an isometry of $(V,\beta)$ then there is a basis for $V$, $\{e_{i},f_{i}\}_{i=1}^{n}$ of $\beta$ hyperbolic pairs so that $\varphi(e_{i})=f_{i}$ and $\varphi(f_{i})=e_{i}$ for all $i=1,\ldots,n$.
\item[(ii)]  If $\varphi$ is induced by an anti-isometry of $(V,\beta)$ then there is a basis for $V$, $\{e_{i},f_{i}\}_{i=1}^{n}$ of $\beta$ hyperbolic pairs so that $\varphi(e_{i})=\alpha f_{i}$ and $\varphi(f_{i})=\alpha^{-1}e_{i}$ for all $i=1,\ldots,n$ where $\alpha$ is a trace 0 element of $\mathbb{F}$.
\end{enumerate}
Conversely any linear transformation of $V$ which satisfies (i) or (ii) induces a flip of $\Delta$.
\end{linflips}

\begin{proof}  From the proof of Lemma \ref{lem.induceflip} it follows that there is a basis of orthogonal $\beta$-hyperbolic pairs $\{h_{i},g_{i}\}_{i=1}^{n}$ so that
\begin{eqnarray*}
\varphi(h_{i})&=&\lambda_{i}g_{i}, \mbox{ and}\\
\varphi(g_{i})&=&\lambda_{i}^{-1}h_{i}
\end{eqnarray*}
for some $\lambda_{i}\in\mathbb{F}$.

\begin{enumerate}
\item[(i)]  Suppose that $\varphi$ is induced by an isometry of $(V,\beta)$.  Since $\beta_{\varphi}$ is $\sigma$-hermitian it follows that for all $i=1,\ldots,n$, $\sigma(\lambda_{i})=\beta_{\varphi}(h_{i},h_{i})\in\mathbb{F}_{q}$ and so in fact $\sigma(\lambda_{i})=\lambda_{i}$.

For $i=1,\ldots,n$ let $g_{i}'=\lambda_{i}g_{i}$.  Choose $\gamma_{i}\in\mathbb{F}$ so that $N_{\sigma}(\gamma_{i})=\lambda_{i}^{-1}$.  Define
\begin{eqnarray*}
e_{i}&=&\gamma_{i}h_{i}, \mbox{ and }\\
f_{i}&=&\gamma_{i}g_{i}'.
\end{eqnarray*}

We now calculate to show that $\{e_{i},f_{i}\}_{i=1}^{n}$ is a basis of $\beta$ hyperbolic pairs with $\varphi(e_{i})=f_{i}$ and $\varphi(f_{i})=e_{i}$.

\begin{eqnarray*}
\beta(e_{i},e_{j})&=&\beta(\gamma_{i}h_{i},\gamma_{j}h_{j})=0 \\
\beta(f_{i},f_{j})&=&\beta(\gamma_{i}g_{i}',\gamma_{i}g_{j}')=\beta(\gamma_{i}\lambda_{i}g_{i},\gamma_{j}\lambda_{j}g_{j})=0 \\
\beta(e_{i},f_{i})&=&\beta(\gamma_{i}h_{i},\gamma_{i}\lambda_{i}g_{i})=N_{\sigma}(\gamma_{i})\lambda_{i}=1 \\
\beta(e_{i},f_{j})&=&\beta(\gamma_{i}h_{i},\gamma_{j}\lambda_{j}g_{j})=0\mbox{ if }i\neq j.
\end{eqnarray*}
Thus $\{e_{i},f_{i}\}_{i=1}^{n}$ forms a $\beta$ hyperbolic basis for $V$, and finally
\begin{eqnarray*}
\varphi(e_{i})=\varphi(\gamma_{i}h_{i})&=&\gamma_{i}\varphi(h_{i})=\gamma_{i}\lambda_{i}g_{i}=\gamma_{i}g_{i}'=f_{i}\\
\varphi(f_{i})=\varphi(\gamma_{i}g_{i}')&=&\gamma_{i}\varphi(g_{i}')=\gamma_{i}\lambda_{i}\varphi(g_{i})=\gamma_{i}h_{i}=e_{i}.\end{eqnarray*}

\item[(ii)]  Suppose now that $\varphi$ is induced by an anti-isometry of $(V,\beta)$.  Since $\beta_{\varphi}$ is $\sigma$-antihermitian it follows that for all $i$, $\sigma(\lambda_{i})=\beta_{\varphi}(h_{i},h_{i})=-\lambda_{i}$, and so $\lambda_{i}$ is of trace 0.  Let $\alpha$ be any non-zero element of trace 0 in $\mathbb{F}$.  For each $i=1,\ldots,n$ choose $a_{i}\in\mathbb{F}_{q}$ so that $a_{i}\lambda_{i}=\alpha$.  Let $\gamma_{i}\in\mathbb{F}$ be chosen so that $N_{\sigma}(\gamma_{i})=a_{i}$.  Set $e_{i}=\gamma_{i}h_{i}$ and $f_{i}=\alpha^{-1}\gamma_{i}\lambda_{i}g_{i}$.  Direct calculation shows that $\{e_{i},f_{i}\}_{i=1}^{n}$ is a $\beta$ hyperbolic basis with the property that $\varphi(e_{i})=\alpha f_{i}$ and $\varphi(f_{i})=\alpha^{-1}e_{i}$.
\end{enumerate}

The converse follows from Lemma \ref{lem.converse}.
\end{proof}

It is now clear that the geometry $\Gamma(n,q)$ depends on whether $\varphi$ is an isometry or an anti-isometry.  With the basis found in Theorem \ref{thm.linearflips} we can see that when $n=1$ the number of points in the geometry depends on whether the flip is an isometry or an anti-isometry, implying that in larger rank the geometries are also not isomorphic.

\subsection{A Flag Transitive Automorphism Group of $\Gamma(n,q)$}

We are interested in finding a group that acts in a natural way on $\Gamma(n,q)$.  The obvious choice for this group is the group of linear transformations of $V$ that preserve both the forms $\beta$ and $\beta_{\varphi}$.  

\begin{defn}  Let $\mathrm{U}_{2n}(q^{2})^{\varphi}=\{f\in\mathrm{U}_{2n}(q^{2})|\beta_{\varphi}(u,v)=\beta_{\varphi}(f(u),f(v))\mbox{ for all }u,v\in V\}$.  
\end{defn}

In this section we will prove three results regarding $\mathrm{U}_{2n}(q^{2})^{\varphi}$.  First we will prove that it is precisely the centralizer in $\mathrm{U}_{2n}(q^{2})$ of $\varphi$.  We will then prove that it acts flag transitively on $\Gamma(n,q)$.  We conclude this section by proving that if $\varphi$ is induced by an isometry of $(V,\beta)$ then $\mathrm{U}_{2n}(q^{2})^{\varphi}\cong\mathrm{U}_{n}(q^{2})\times\mathrm{U}_{n}(q^{2})$ and if $\varphi$ is induced by an anti-isometry then $\mathrm{U}_{2n}(q^{2})^{\varphi}\cong \mathrm{GL}_{n}(q^{2})$.

\begin{lem}\label{lem.lcommutator}  $\mathrm{U}_{2n}(q^{2})^{\varphi}=C_{\Gamma\mathrm{U}_{2n}(q^{2})}(\varphi)\cap\mathrm{U}_{2n}(q^{2})$.
\end{lem}

\begin{proof}  Let $f\in\mathrm{U}_{2n}(q^{2})^{\varphi}$ and $v\in V$.  To check that $\varphi$ commutes with $f$ we will show that for all $w\in V$, $\beta(w,f(\varphi(v)))=\beta(w,\varphi(f(v)))$.  Since $\beta$ is non-degenerate this will force $f(\varphi(v))=\varphi(f(v))$.

Let $w\in V$.  Choose $x\in V$ so that $f(x)=w$.  Then
\begin{eqnarray*}
\beta_{\varphi}(x,v)=\beta(x,\varphi(v))=\beta(f(x),f(\varphi(v)))&=&\beta(w,f(\varphi(v)))\mbox{ and }\\
\beta_{\varphi}(x,v)=\beta_{\varphi}(f(x),f(v))=\beta(f(x),\varphi(f(v)))&=&\beta(w,\varphi(f(v))).
\end{eqnarray*}

Conversely if $f\in\mathrm{U}_{2n}(q^{2})$ commutes with $\varphi$ then for all $u$, $v\in V$ we have
$$
\beta_{\varphi}(f(u),f(v))=\beta(f(u),\varphi(f(v)))=\beta(f(u),f(\varphi(v)))=\beta(u,\varphi(v))=\beta_{\varphi}(u,v).\qedhere
$$
\end{proof}

We now turn to the problem of showing that $\mathrm{U}_{2n}(q^{2})^{\varphi}$ acts flag transitively on $\Gamma(n,q)$.  Before we can prove that $\mathrm{U}_{2n}(q^{2})^{\varphi}$ acts flag transitively on $\Gamma(n,q)$ we require one more lemma.

\begin{lem}\label{lem.linflipnicestbasis}  Let $C=(C_{i})_{i=1}^{n}$ be a chamber of $\Gamma(n,q)$.  
\begin{enumerate}
\item[(a)]  If $\varphi$ is induced by an isometry of $(V,\beta)$ as in Lemma \ref{lem.induceflip} then there is a basis $\{e_{i},f_{i}\}_{i=1}^{n}$ for $V$ with the following properties:
\begin{enumerate}
\item[(i)]  $\{e_{i},f_{i}\}_{i=1}^{n}$ is hyperbolic with respect to $\beta$;
\item[(ii)]  for all $i=1,\ldots,n$, $\varphi(e_{i})=f_{i}$ and $\varphi(f_{i})=e_{i}$; and
\item[(iii)]  for all $i=1,\ldots,n$, $C_{i}=\langle e_{1},\ldots,e_{i}\rangle$.
\end{enumerate}
\item[(b)]  If $\varphi$ is induced by an anti-isometry of $(V,\beta)$ as in Lemma \ref{lem.induceflip} then there is a basis $\{e_{i},f_{i}\}_{i=1}^{n}$ for $V$ with properties (i) and (iii), and
\begin{enumerate}
\item[(ii')]  for all $i=1,\ldots,n$, $\varphi(e_{i})=\alpha f_{i}$ and $\varphi(f_{i})=\alpha^{-1}e_{i}$ where $\alpha$ is any non-zero element of trace 0 in $\mathbb{F}$.
\end{enumerate}
\end{enumerate}
\end{lem}

\begin{proof}  
\begin{enumerate}
\item[(a)]  Let $e_{1}$ be a non-zero vector in $C_{1}$.  Then after scaling as in the proof of Main Theorem 1A(i) we may assume that $(e_{1},\varphi(e_{1}))$ is a hyperbolic pair.  Since $C_{1}^{\dperp}\cap C_{2}$ is an element of $\Gamma(n,q)$ by Lemma \ref{lem.complementingamma} we can choose $e_{2}\in C_{1}^{\dperp}\cap C_{2}$ so that after scaling, $(e_{2},\varphi(e_{2}))$ is a hyperbolic pair.  Repeating this procedure we produce the desired basis.
\item[(b)]  This is proved in the same fashion of (a), with the scaling as in the proof of Main Theorem 1A(ii) replacing the scaling in Main Theorem 1A(i).\qedhere
\end{enumerate}
\end{proof}

\begin{lft}  If $\varphi$ is a linear flip then $\mathrm{U}_{2n}(q^{2})^{\varphi}$ acts flag transitively on $\Gamma(n,q)$.
\end{lft}

\begin{proof}  Since $\Gamma(n,q)$ is a geometry, it suffices to show that $\mathrm{U}_{2n}(q^{2})^{\varphi}$ acts chamber transitively.  Let $C=(C_{i})_{i=1}^{n}$ and $D=(D_{i})_{i=1}^{n}$ be two chambers of $\Gamma(n,q)$.  By Lemma \ref{lem.linflipnicestbasis} we can find bases $\{e_{i},f_{i}\}_{i=1}^{n}$ and $\{g_{i},h_{i}\}_{i=1}^{n}$ for $C$ and $D$ respectively such that if $\varphi$ is induced by an isometry as in Lemma \ref{lem.induceflip}, both bases satisfy (a), and if $\varphi$ is induced by any anti-isometry as in Lemma \ref{lem.induceflip}, both bases satisfy (b) for the same $\alpha\in\mathbb{F}$ of trace 0.  Notice that $\{e_{i},\varphi(e_{i})\}_{i=1}^{n}$ and $\{g_{i},\varphi(g_{i})\}_{i=1}^{n}$ form bases for $V$.  Furthermore, the Gram matrix for $\beta$ is the same whether we take the basis $\{e_{i},\varphi(e_{i})\}_{i=1}^{n}$ or the basis $\{g_{i},\varphi(g_{i})\}_{i=1}^{n}$.  Similarly the Gram matrix for $\beta_{\varphi}$ does not depend on which of these two bases we consider.

Define $T:V\to V$ by $T(e_{i})=g_{i}$ and $T(\varphi(e_{i}))=\varphi(g_{i})$ and extend linearly.  It is easy to see that $T$ preserves $\beta$ and commutes with $\varphi$, and hence also preserves $\beta_{\varphi}$.  Thus $T$ is an element of $\mathrm{U}_{2n}(q^{2})^{\varphi}$ with $T(C)=D$.
\end{proof}

\begin{thm}  Let $\varphi$ be a linear flip of $\Delta$.
\begin{enumerate}
\item[(i)]  If $\varphi$ is induced by an isometry of $(V,\beta)$ as in Lemma \ref{lem.induceflip} then $\mathrm{U}_{2n}(q^{2})^{\varphi}\cong\mathrm{U}_{n}(q^{2})\times\mathrm{U}_{n}(q^{2})$.
\item[(ii)]  If $\varphi$ is induced by an anti-isometry of $(V,\beta)$ as in Lemma \ref{lem.induceflip} then $\mathrm{U}_{2n}(q^{2})^{\varphi}\cong\mathrm{GL}_{n}(q^{2})$.
\end{enumerate}
\end{thm}

\begin{proof}  Let $\{e_{i},f_{i}\}_{i=1}^{n}$ be a basis for $V$ as in Main Theorem 1A.
\begin{enumerate}
\item[(i)]  Define a new basis for $V$ by $g_{i}=e_{i}+f_{i}$ for $i=1,\ldots,n$ and $h_{i}=e_{i}-f_{i}$ for $i=1,\ldots,n$.  Order this basis as $\{g_{1},\ldots,g_{n},h_{1},\ldots,h_{n}\}$.  Direct calculation shows that with respect to this (ordered) basis, $\beta$ and $\beta_{\varphi}$ have Gram matrices
$$
M_{1}=\left(\begin{array}{cc}
2I_{n} & 0 \\
0 & -2I_{n}\\
\end{array}\right)\mbox{ and }M_{2}=\left(\begin{array}{cc}
2I_{n} & 0 \\
0 & 2I_{n} \\
\end{array}\right)
$$
respectively.
Given a linear transformation $T$ of $V$, we can express $T$ as a block matrix
$$
T=\left(\begin{array}{cc}
A & B \\
C & D \\
\end{array}\right).
$$
Since both $\beta$ and $\beta_{\varphi}$ are hermitian, it follows that $T$ preserves both forms if and only if for $i=1,2$, $\sigma(T^{t})M_{i}T=M_{i}$ where $T^{t}$ denotes the transpose of $T$.  These two requirements are by direct calculation equivalent to the following four equalities:
\begin{eqnarray}
\sigma(A^{t})A-\sigma(C^{t})C=\sigma(D^{t})D-\sigma(B^{t})B&=& I_{n};\\
\sigma(A^{t})A+\sigma(C^{t})C=\sigma(B^{t})B+\sigma(D^{t})D&=&I_{n};\\
\sigma(A^{t})B-\sigma(C^{t})D=\sigma(B^{t})A-\sigma(D^{t})C&=&0;\\
\sigma(A^{t})B+\sigma(C^{t})D=\sigma(B^{t})A+\sigma(D^{t})C&=&0.
\end{eqnarray}
Adding (3) to (4) shows that $\sigma(A^{t})A=\sigma(D^{t})D=I_{n}$, and so $A$ and $D$ are unitary matrices.  Adding (5) to (6) and using the fact that $A$ is invertible shows that $B=0$.  Similarly subtracting (5) from (6) and using the fact that $D$ is invertible shows that $C=0$.  Thus in fact
$$
T=\left(\begin{array}{cc}
A & 0 \\
0 & D \\
\end{array}\right)
$$
where $A$ and $D$ are unitary matrices.  Conversely it is easy to check that if $A$ and $D$ are unitary matrices then 
$$
\left(\begin{array}{cc}
A & 0 \\
0 & D \\
\end{array}\right)
$$
preserves both $\beta$ and $\beta_{\varphi}$, and so lies in $\mathrm{U}_{2n}(q^{2})^{\varphi}$.

\item[(ii)]  The technique here is the same as in (i), but the details are different.  We only outline this part of the proof.  Define a new basis for $V$ by setting $g_{i}=e_{i}+\alpha f_{i}$ for $i=1,\ldots,n$ and $h_{i}=e_{i}-\alpha f_{i}$ for $i=n,\ldots,n$ and order this basis $\{g_{1},\ldots,g_{n},h_{1},\ldots,h_{n}\}$.  Considering a linear transformation of $V$ as a block matrix $T$ as above, direct calculation shows that $T$ preserves both $\beta$ and $\beta_{\varphi}$ if and only if $B=C=0$ and $\sigma(A^{t})D=I_{n}$.  Conversely for any $A\in\mathrm{GL}_{n}(q^{2})$ the matrix
$$
\left(\begin{array}{cc}
A & 0 \\
0 & (\sigma(A^{t}))^{-1}
\end{array}\right)
$$
preserves both $\beta$ and $\beta_{\varphi}$ and so lies in $\mathrm{U}_{2n}(q^{2})^{\varphi}$.\qedhere
\end{enumerate}
\end{proof}

\section{$\sigma$-Semilinear Flips}

We now turn our attention to the study of $\sigma$-semilinear flips of the unitary building.  In particular we prove in Theorem \ref{thm.classificationofflips} that there are only 2 similarity classes of $\sigma$-semilinear flips of the unitary building.

Throughout this section, $\varphi$ denotes a $\sigma$-semilinear flip.  Recall that we have identified the flip with a $\sigma$-semilinear isometry $f\in\Gamma\mathrm{U}(V)$ that induces $\varphi$ and satisfies $f^{2}=\mathrm{id}$ on $V$.

\subsection{Geometries Induced by a $\sigma$-semilinear Flip}\label{sec.sgeom}

We've already seen that the geometry induced by a flip $\varphi$ is related to a form $\beta_{\varphi}$ defined on $V$.  In the case of a linear flip we saw this form is $\sigma$-hermitian or $\sigma$-antihermitian.  As we saw in Lemma \ref{lem.sigmaform}, when $\varphi$ is a $\sigma$-semilinear flip of the unitary building, the induced form $\beta_{\varphi}$ is symmetric.  

\begin{defn}  Let $U$ be a subspace of $V$ and let $\mathcal{U}$ be a basis for $V$.  Recall that $\beta_{\varphi}(\mathcal{U})$ denotes the Gram matrix of the form $\beta_{\varphi}$ restricted to $U$ with respect to the basis $\mathcal{U}$.  The \textbf{discriminant} of $U$ is:
$$
\mathrm{disc}(U)=\left\{\begin{array}{r l}
1, & \mbox{ if }\det(\beta_{\varphi}(\mathcal{U}))\mbox{ is a square in }\mathbb{F};\\
-1, & \mbox{ if }\det(\beta_{\varphi}(\mathcal{U}))\mbox{ is a non-square in }\mathbb{F};\\
0, & \mbox{ if $U$ is $\beta_{\varphi}$ degenerate.}
\end{array}
\right.
$$
\end{defn}

\begin{defn}  A \textbf{square type} (resp. \textbf{non-square type}) $i$-space is an $i$-dimensional subspace $U$ of $V$ with $\mathrm{disc}(U)=1$ (resp. $\mathrm{disc}(U)=-1$).
\end{defn}

\begin{lem}\label{lem.splitdisc}  Let $U$, $U'\in\Gamma(n,q)$ with $U\subseteq U'$.  Let $W=\langle U,\varphi(U)\rangle^{\perp}\cap U'$.  Then
$$
\mathrm{disc}(U')=\mathrm{disc}(U)\mathrm{disc}(W).
$$
\end{lem}

\begin{proof}  Since $W$ is a $\beta_{\varphi}$ orthogonal complement to $U$ in $U'$, we can choose a basis relative to this decomposition and so represent $\beta_{\varphi}(U')$ as a matrix with the form
$$
\beta_{\varphi}(U')=\left(
                   \begin{array}{cc}
                     \beta_{\varphi}(W) & 0 \\
                     0 & \beta_{\varphi}(U) \\
                   \end{array}
                 \right)
$$
which has determinant $(\det\beta_{\varphi}(W))(\det\beta_{\varphi}(U))$.
\end{proof}

We now define two pregeometries contained in $\Gamma(n,q)$.  We will prove shortly that these are in fact geometries and in Section \ref{sec.slinflag} we show that both these geometries are flag transitive.

\begin{defn}  Define the following pregeometries in $\Gamma(n,q)$:
\begin{eqnarray*}
\Gamma_{1}(n,q)&=&\{U\in\Gamma(n,q)|\mathrm{disc}(U)=1\mbox{ or }\dim(U)=n\};\\
\Gamma_{-1}(n,q)&=&\{U\in\Gamma(n,q)|\mathrm{disc}(U)=-1\mbox{ or }\dim(U)=n\}.
\end{eqnarray*}
\end{defn}

\begin{note}  Just as in the study of $\Gamma(n,q)$, a \textbf{point} of $\Gamma_{1}(n,q)$ (resp. $\Gamma_{-1}(n,q)$) is a 1-dimensional object of $\Gamma_{1}(n,q)$ (resp. $\Gamma_{-1}(n,q)$) and a \textbf{line} of $\Gamma_{1}(n,q)$ (resp. $\Gamma_{-1}(n,q)$) is a 2-dimensional object of $\Gamma_{1}(n,q)$ (resp. $\Gamma_{-1}(n,q)$).
\end{note}

It is not immediately clear why we take all $n$-dimensional elements of $\Gamma(n,q)$ in both $\Gamma_{1}(n,q)$ and $\Gamma_{-1}(n,q)$.  We will see shortly (Theorem \ref{thm.maxtype}) that the $n$-dimensional objects of $\Gamma(n,q)$ all have the same $\beta_{\varphi}$ type.  In order that $\Gamma_{1}(n,q)$ and $\Gamma_{-1}(n,q)$ are both geometries of rank $n$, we must include these $n$-dimensional objects.

We now explore the properties of the geometries $\Gamma(n,q)$, $\Gamma_{1}(n,q)$ and $\Gamma_{-1}(n,q)$ in some more detail.  Looking at these geometries will give insight into the structure of the flip.  The next two results admit a uniform proof.

\begin{lem}\label{lem.lineform}  If $L$ is a line of $\Gamma(n,q)$ then there are biorthogonal points $u$, $v$ of $\Gamma(n,q)$ so that $L=\langle u,v\rangle$.  If $L$ is a line of $\Gamma_{1}(n,q)$ we can assume both $u$ and $v$ are points of $\Gamma_{1}(n,q)$.  If $L$ is a line of $\Gamma_{-1}(n,q)$ then one of $u$ or $v$ is of square type and the other is of non-square type.
\end{lem}

\begin{cor}\label{cor.bothtypes}  Let $L$ be a line of $\Gamma(n,q)$.  Then $L$ contains points of both $\Gamma_{1}(n,q)$ and $\Gamma_{-1}(n,q)$.
\end{cor}

\begin{proof}[Proof of Lemma \ref{lem.lineform} and Corollary \ref{cor.bothtypes}]   Since $L$ is already $\beta$ totally isotropic it suffices to only consider the form $\beta_{\varphi}$.  The results follow by straightforward calculations.
\end{proof}

\begin{lem}\label{lem.pointobjgamma}  If $\dim U\geq 2$ and $U\in\Gamma(n,q)$, then $U$ contains points of both $\Gamma_{1}(n,q)$ and $\Gamma_{-1}(n,q)$.
\end{lem}

\begin{cor}\label{cor.pointobj1}  If $U\in\Gamma_{\epsilon}(n,q)$ then $U$ contains a point of $\Gamma_{\epsilon}(n,q)$.
\end{cor}

\begin{proof}  If $\dim(U)=1$ then $U$ is a point of $\Gamma_{\epsilon}(n,q)$.  If $\dim(U)>1$ then we can apply Lemma \ref{lem.pointobjgamma} to conclude that $U$ contains a point of $\Gamma_{\epsilon}(n,q)$.
\end{proof}

\begin{cor}\label{cor.kflag}  A $h$-object $U$ of $\Gamma_{1}(n,q)$ with $h<n$ has a basis $\{u_{1},\ldots,u_{h}\}$ of square type points that are pairwise biorthogonal.  Consequently any $h$-object $U$ of $\Gamma_{1}(n,q)$ contains a $\{1,\ldots,h\}$ flag of $\Gamma_{1}(n,q)$.
\end{cor}

\begin{proof}  The conclusion follows by induction on $h$ as in the proof of Corollary \ref{cor.basebiorthopoints} with Corollary \ref{cor.pointobj1} replacing Lemma \ref{lem.gammapointexist} and Lemma \ref{lem.splitdisc} replacing Lemma \ref{lem.complementingamma}. 
\end{proof}

\begin{cor}\label{cor.bigbasis}  Let $M$ be an $n$-object of $\Gamma(n,q)$.  If $M$ is of square type then $M$ has a basis $\{u_{1},\ldots,u_{n}\}$ of pairwise biorthogonal square type points.  If $M$ is of non-square type them $M$ has a basis $\{v_{1},\ldots,v_{n}\}$ of pairwise biorthogonal points, where for $i=1,\ldots,n-1$, $v_{i}$ is of square type and $v_{n}$ is of non-square type.
\end{cor}

\begin{proof}  If $M$ is of square type it contains a square type point $u_{1}$.  By applying Corollary \ref{cor.kflag} to $M'=M\cap \langle u_{1}\rangle^{\dperp}$ we produce the remaining points.  If $M$ is of non-square type it contains a non-square type point $v_{n}$ and applying Corollary \ref{cor.kflag} to $M'=\langle v_{n}\rangle^{\dperp}$ produces the remaining points.
\end{proof}

\begin{rem}   We will prove in Theorem \ref{thm.maxtype} that for a fixed flip $\varphi$, every maximal object has the same $\beta_{\varphi}$ type.
\end{rem}

\begin{lem}\label{lem.gamma1pointresidue}  Let $u$ be a point of $\Gamma_{\epsilon}(n,q)$ for $\epsilon\in\{1,-1\}$.  Then $\mathrm{res}_{\Gamma_{\epsilon}(n,q)}(u)\cong\Gamma_{1}(n-1,q)$.
\end{lem}  

\begin{proof}  This follows immediately from the proof of Corollary \ref{cor.gammaresidue} once we note that if $u\in\Gamma_{1}(n,q)$  then by Lemma \ref{lem.splitdisc}, $r_{u}$ sends square type subspaces to square type subspaces, and non-square type subspaces to non-square type subspaces.  Similarly if $u\in\Gamma_{-1}(n,q)$ then $r_{u}$ sends square type subspaces to non-square type subspaces, and non-square type subspaces to square type subspaces.
\end{proof}

\begin{thm}\label{thm.gamma1geometry}  $\Gamma_{1}(n,q)$ and $\Gamma_{-1}(n,q)$ are a geometries with type and incidence inherited from $\Gamma(n,q)$.
\end{thm}

\begin{proof}  The same arguments as in the proof of Theorem \ref{thm.biggeometry} work in this case, with Lemma \ref{lem.gamma1pointresidue} replacing Corollary \ref{cor.gammaresidue} and Corollary \ref{cor.pointobj1} replacing Lemma \ref{lem.gammapointexist}.
\end{proof}

\begin{lem}\label{lem.2dvarphiinv}  Let $u$ be a point of $\Gamma(n,q)$ and let $U=\langle u,\varphi(u)\rangle$.  Then either every point of $\Gamma(n,q)$ contained in $U$ is of square type, or every point is of non-square type.
\end{lem}

\begin{proof}  We first show that if $\lambda$ has trace 0 in $\mathbb{F}$ then $1+\lambda^{2}$ lies in $\mathbb{F}_{q}$ and so in particular is a square in $\mathbb{F}$.  Since $\mathrm{Tr}_{\sigma}(\lambda)=0$ it follows that $\sigma(\lambda)=-\lambda$ and so $\sigma(\lambda^{2})=\sigma(\lambda)^{2}=\lambda^{2}$.  Hence $\lambda^{2}\in\mathbb{F}_{q}$ and thus also $1+\lambda^{2}\in\mathbb{F}$.

If $u$ is a point in $U$, then all the other points of $\Gamma(n,q)$ that lie in $U$ (except for $\varphi(u)$) are of the form $u+\lambda\varphi(u)$ for some non-zero $\lambda$ of trace 0, and since these points have $Q_{\varphi}$ value $(1+\lambda^{2})Q_{\varphi}(u)$, we conclude that all these points have the same type as $u$.  Since also $\varphi(u)$ has the same $\beta_{\varphi}$ type as $u$ we conclude that all the points of $U$ have the same $\beta_{\varphi}$ type.
\end{proof}

\begin{thm}\label{thm.maxtype}  Let $M$ and $M'$ be $n$-objects of $\Gamma(n,q)$.  Then $M$ and $M'$ have the same $\beta_{\varphi}$ type.
\end{thm}

\begin{proof}  If $n=1$ then this follows because $V=\langle u,\varphi(u)\rangle$ for some point $u$ of $\Gamma(1,q)$.  Since every point on $\langle u,\varphi(u)\rangle$ has the same $\beta_{\varphi}$ type as $u$ (by Lemma \ref{lem.2dvarphiinv}), when $n=1$ the maximal objects of $\Gamma(1,q)$ all have the same $\beta_{\varphi}$ type.

Assume $n>1$, suppose $M$ is of square type and suppose $M'$ is of non-square type.   Then by Corollary \ref{cor.bigbasis} $M$ has a basis $\{e_{1},\ldots,e_{n}\}$ of biorthogonal square type points.  We may then scale each $e_{i}$ so that $\beta_{\varphi}(e_{i},e_{i})=1$ for all $i$.  Setting $f_{i}=\varphi(e_{i})$ for $i=1,\ldots,n$ we obtain a $\beta$ hyperbolic basis $\mathcal{B}_{1}=\{e_{i},\varphi(e_{i})\}_{i=1}^{n}$ for $V$.

Similarly by Corollary \ref{cor.bigbasis} $M'$ has a basis $\{g_{1},\ldots,g_{n}\}$ of biorthogonal points where $g_{1},\ldots,g_{n-1}$ are of square type and $g_{n}$ is of non-square type.  After scaling we can assume that $\beta_{\varphi}(g_{i},g_{i})=1$ for $i=1,\ldots,n-1$, and $\beta_{\varphi}(g_{n},g_{n})=\alpha$ where $\alpha$ is a non-square in $\mathbb{F}$.  Let $h_{n}=\alpha^{-1}g_{n}$.  Notice that $\varphi(h_{n})=\sigma(\alpha^{-1})\varphi(g_{n})$ and so $\varphi(g_{n})=\sigma(\alpha)\varphi(h_{n})$.  It is easy to see that $(h_{n},\varphi(g_{n}))$ forms a $\beta$ hyperbolic pair, and that $\beta_{\varphi}(h_{n},h_{n})=\alpha^{-1}$, and $\beta_{\varphi}(\varphi(g_{n}),\varphi(g_{n}))=\sigma(\alpha)$.  This gives another $\beta$ hyperbolic basis, $\mathcal{B}_{2}=\{g_{1},\varphi(g_{1}),\ldots,g_{n-1},\varphi(g_{n-1}),h_{n},\sigma(\alpha)\varphi(h_{n})\}$.

Let $T$ be the transition matrix from $\mathcal{B}_{1}$ to $\mathcal{B}_{2}$.  For $i=1,2$ let $B_{i}$ denote the Gram matrix of $\beta$ with respect to $\mathcal{B}_{i}$, and let $C_{i}$ denote the Gram matrix of $\beta_{\varphi}$ with respect to $\mathcal{B}_{i}$.  Since $T$ is the transition matrix from $\mathcal{B}_{1}$ to $\mathcal{B}_{2}$ it follows that
\begin{eqnarray}
TB_{1}\sigma(T^{t})&=&B_{2} \\
TC_{1}T^{t}&=&C_{2}.
\end{eqnarray}
Since $C_{1}=I_{2n}$ it follows from (8) that $TT^{t}=C_{2}$.  It is easy to check that $\det(C_{2})=\alpha^{q-1}$ and so also $\det(T)^{2}=\alpha^{q-1}$.  Hence $\det(T)=\pm\alpha^{(q-1)/2}$ and it is easy to check that this implies $N_{\sigma}(\det(T))=-1$.

On the other hand, since both $\mathcal{B}_{1}$ and $\mathcal{B}_{2}$ are $\beta$ hyperbolic bases, we see that $B_{1}=B_{2}$, which combined with (7) forces $N_{\sigma}(\det(T))=1$, a contradiction.  Hence $M$ and $M'$ have the same $\beta_{\varphi}$ type.
\end{proof}

\begin{defn}  A $\sigma$-semilinear flip $\varphi$ is of \textbf{square type} if the maximal objects of $\Gamma(n,q)$ are of square $\beta_{\varphi}$ type.  A $\sigma$-semilinear flip $\varphi$ is of \textbf{non-square type} if the maximal objects of $\Gamma(n,q)$ are of non-square $\beta_{\varphi}$ type.
\end{defn}

\subsection{Classification of $\sigma$-Semilinear Flips of the Unitary Building}\label{sec.slinflips}

We are now in a position to fully classify the flips of the unitary building that are induced by $\sigma$-semilinear transformations of the underlying vector space. 

\begin{slinflips}\label{thm.classificationofflips}  Let $\varphi$ be a $\sigma$-semilinear flip on the unitary building.  Then there is a basis for $V$, $\{e_{i},f_{i}\}_{i=1}^{n}$ of $\beta$ hyperbolic pairs so that for $i=1,\ldots,n-1$ we have $\varphi(e_{i})=f_{i}$, $\varphi(f_{i})=e_{i}$ and either
\begin{enumerate}
\item[(i)]  $\varphi(e_{n})=f_{n}$, $\varphi(f_{n})=e_{n}$ or
\item[(ii)]  $\varphi(e_{n})=\lambda f_{n}$, $\varphi(f_{n})=\sigma(\lambda^{-1})e_{n}$ where $\lambda$ is a non-square in $\mathbb{F}$.
\end{enumerate}
Case (i) occurs if $\varphi$ is of square type, and Case (ii) occurs if $\varphi$ is of non-square type.  Conversely any $\sigma$-semilinear transformation of $V$ that satisfies the conclusion of this theorem induces a flip of $\Delta$.
\end{slinflips}

\begin{proof}  The forward implication follows immediately from the proof of Theorem \ref{thm.maxtype}.  That is, the proof of Theorem \ref{thm.maxtype} shows that if $\varphi$ is of square type then there is a basis as described by (i), and if $\varphi$ is of non-square type then there is a basis as described in (ii).

The converse follows from Lemma \ref{lem.converse}.
\end{proof}

This completes the proof of Main Theorem 1.
It is worth noting that the geometry of $\beta$ isotropic $\beta_{\varphi}$ square-type subspaces induced by square-type and non-square type flips are not isomorphic, as can be seen from the fact that square-type and non-square type spaces contain different numbers of square-type points.

\subsection{A Flag Transitive Automorphism Group of $\Gamma_{1}(n,q)$}\label{sec.slinflag}

We now study the group of linear transformations of $V$ that preserve both $\beta$ and $\beta_{\varphi}$.  This group acts as an automorphism group of the geometry $\Gamma(n,q)$ although it does not act flag transitively on that geometry.  We prove in this section that this group acts flag transitively on both $\Gamma_{1}(n,q)$ and $\Gamma_{-1}(n,q)$.

\begin{defn}  Let $\mathrm{U}_{2n}(q^{2})^{\varphi}=\{f\in\mathrm{U}_{2n}(q^{2})|\beta_{\varphi}(u,v)=\beta_{\varphi}(f(u),f(v))\mbox{ for all }u,v\in V\}$.  Notice that this is the same notation we used in the case of a linear flip.
\end{defn}

\begin{lem}\label{lem.semicommutator}  $\mathrm{U}_{2n}(q^{2})^{\varphi}=C_{\Gamma\mathrm{U}(V)}(\varphi)\cap\mathrm{U}_{2n}(q^{2})$.
\end{lem}

\begin{proof}  The same proof as in Lemma \ref{lem.lcommutator} holds in this case.
\end{proof}

\begin{lem}\label{lem.seminicebase}  Let $C=(C_{i})_{i=1}^{n}$ be a chamber of $\Gamma_{1}(n,q)$.
\begin{enumerate}
\item[(a)]  If $\varphi$ is of square type then there is a basis $\{e_{i},f_{i}\}_{i=1}^{n}$ for $V$ with the following properties:
\begin{enumerate}
\item[(i)]  $\{e_{i},f_{i}\}$ is hyperbolic with respect to $\beta$;
\item[(ii)]  for all $i=1,\ldots,n$, $\varphi(e_{i})=f_{i}$ and $\varphi(f_{i})=e_{i}$; and
\item[(iii)]  for all $i=1,\ldots,n$, $C_{i}=\langle e_{1},\ldots,e_{i}\rangle$.
\end{enumerate}
\item[(b)]  If $\varphi$ is of non-square type then there is a basis $\{e_{i},f_{i}\}_{i=1}^{n}$ for $V$ with properties (i) and (iii) and
\begin{enumerate}
\item[(ii')]  For $i=1,\ldots,n-1$ $\varphi(e_{i})=f_{i}$ and $\varphi(f_{i})=e_{i}$, but $\varphi(e_{n})=\lambda f_{n}$ and $\varphi(f_{n})=\sigma(\lambda^{-1})e_{n}$ where $\lambda$ is any non-square in $\mathbb{F}$.
\end{enumerate}
\end{enumerate}
\end{lem}

\begin{proof}  The same arguments as in the proof of Lemma \ref{lem.linflipnicestbasis} works with Corollary \ref{cor.pointobj1} replacing Lemma \ref{lem.gammapointexist} and the scaling arguments from the proof of Theorem \ref{thm.maxtype} replacing the scaling arguments from the proof of Main Theorem 1A.
\end{proof}

Lemma \ref{lem.seminicebase} is strictly stronger than Main Theorem 1B.  In Main Theorem 1B we only prove that, given a maximal object we can find a basis satisfying (i) and (ii) (resp. (ii')), what is interesting about Lemma \ref{lem.seminicebase} is that given any chamber of $\Delta^{\varphi}$ we can additionally require that (iii) be satisfied.

\begin{slft}\label{thm.flagtransitive1}  $\mathrm{U}_{2n}(q^{2})^{\varphi}$ acts flag transitively on $\Gamma_{1}(n,q)$.
\end{slft}

\begin{proof}  The same proof as in Main Theorem 2 applies here, with Lemma \ref{lem.seminicebase} replacing Lemma \ref{lem.linflipnicestbasis}, but note that we must choose the $\lambda$'s to be the same if $\varphi$ is a non-square type flip.
\end{proof}

The proof that $\mathrm{U}_{2n}(q^{2})^{\varphi}$ acts flag transitively on $\Gamma_{-1}(n,q)$ is similar.

We now turn to the problem of determining the isomorphism type of $\mathrm{U}_{2n}(q^{2})^{\varphi}$ when $\varphi$ is a $\sigma$-semilinear flip of $\Delta$.  The following lemma is well known.

\begin{lem}\label{lem.determinetype}  Let $(W,\rho)$ be a non-degenerate orthogonal space over a field $k$.  Let $U$ be a 2-dimensional non-degenerate subspace of $W$.
\begin{enumerate}
\item[(i)]  If $-1$ is a square in $k$, then $U$ is of $+$ type if and only if $U$ is of square type.
\item[(ii)]  If $-1$ is not a square in $k$, then $U$ is of $+$ type if and only if $U$ is of non-square type.
\end{enumerate}
\end{lem}

\begin{thm}\label{thm.grameqbases}  There is a basis for $V$ relative to which the Gram-matrices of $\beta$ and $\beta_{\varphi}$ coincide.
\end{thm}

\begin{proof}  Choose a basis for $V$ as provided by Main Theorem 1.  Then we have a basis $\{e_{i},f_{i}\}_{i=1}^{n}$ so that each $(e_{i},f_{i})$ is a $\beta$ hyperbolic pair, and for $i=1,\ldots,n-1$ we have
$$
\begin{array}{lcl}
\varphi(e_{i})=f_{i} & \mbox{ and } &   \varphi(e_{n})=\lambda f_{n} \\
\varphi(f_{i})=e_{i} &                       &    \varphi(f_{n})=\sigma(\lambda^{-1})e_{n}
\end{array}
$$

If $\varphi$ is a square type flip then we may assume that $\lambda=1$.  If $\varphi$ is a non-square type flip then $\lambda$ is a non-square in $\mathbb{F}$.

Choose $a\in\mathbb{F}_{q}$ that is not a square (in $\mathbb{F}_{q}$) and choose $\alpha\in\mathbb{F}$ so that $\alpha^{2}=a$.  Notice that $\sigma(\alpha)=-\alpha$.  Define a new basis as follows.  For $i=1,\ldots,n-1$ set
\begin{eqnarray*}
g_{i}&=&e_{i}+f_{i} \\
g_{i+n}&=& \alpha(e_{i}-f_{i}) \\
g_{n}&=& e_{n}+\lambda f_{n} \\
g_{2n}&=& \alpha(e_{n}-\lambda f_{n}) \\
\end{eqnarray*}
It is easy to check that each vector $g_{i}$, $i=1,\ldots,2n$,  is fixed by $\varphi$, and so since $\varphi$ acts trivially on the basis $\{g_{i}|i=1,\ldots,2n\}$ we conclude the Gram matrices of $\beta$ and $\beta_{\varphi}$ agree with respect to this basis.
\end{proof}

\begin{cons}\label{cons.suitablebasis}   For convenience, we reorder the basis $\{g_{i}|i=1,\ldots,2n\}$ from the proof of Theorem \ref{thm.grameqbases} as follows:  for $i=1,\ldots,n$ set
\begin{eqnarray*}
h_{2i-1}&=&g_{i}\\
h_{2i}&=&g_{i+n}.
\end{eqnarray*}

With respect to the basis $\{h_{i}\}_{i=1}^{2n}$, the common Gram matrix of $\beta$ and $\beta_{\varphi}$ is a block diagonal matrix with $2\times 2$ blocks $\{M_{i}\}_{i=1}^{n}$.  If $\varphi$ is of square type then for $i=1,\ldots,n$ we have
$$
M_{i}=\left(
        \begin{array}{cc}
          2 & 0 \\
          0 & 2a \\
        \end{array}
      \right).
$$
If $\varphi$ is of non-square type then for $i=1,\ldots,n-1$ we have the same matrix $M_{i}$ as above and
$$
M_{n}=\left(
      \begin{array}{cc}
         \mathrm{Tr}_{\sigma}(\lambda) & \alpha(\sigma(\lambda)-\lambda) \\
         \alpha(\sigma(\lambda)-\lambda) & a\mathrm{Tr}_{\sigma}(\lambda) \\
      \end{array}
    \right).
$$
\end{cons}

\begin{lem}\label{lem.nsqnorm}  Suppose $\gamma$ is a non-square in $\mathbb{F}$.  Then $N_{\sigma}(\gamma)$ is a non-square in $\mathbb{F}_{q}$.
\end{lem}

\begin{proof}  Since the norm map $N_{\sigma}$ is multiplicative it suffices to show that if $\alpha$ is a generator of $\mathbb{F}^{\ast}$ then $N_{\sigma}(\alpha)$ is a non-square in $\mathbb{F}_{q}$.  This follows from a straightforward calculation which shows that the square roots of $N_{\sigma}(\alpha)$ in $\mathbb{F}$ are not fixed by $\sigma$, and so do not lie in $\mathbb{F}_{q}$.
\end{proof}

\begin{thm}\label{thm.flipgroup}  Let $\varphi$ be a $\sigma$-semilinear flip of $\Delta$.
\begin{enumerate}
\item[(1)]  Suppose $\varphi$ is a square type flip.
\begin{enumerate}
\item[(i)]  If $n$ is even or $-1$ is not a square in $\mathbb{F}_{q}$ then $\mathrm{U}_{2n}(q^{2})^{\varphi}\cong\mathrm{O}_{2n}^{+}(q)$.
\item[(ii)]  If $n$ is odd and $-1$ is a square in $\mathbb{F}_{q}$ then $\mathrm{U}_{2n}(q^{2})^{\varphi}\cong\mathrm{O}_{2n}^{-}(q)$.
\end{enumerate}
\item[(2)]  Suppose $\varphi$ is a non-square type flip.
\begin{enumerate}
\item[(i)]  If $n$ is even or $-1$ is not a square in $\mathbb{F}_{q}$ then $\mathrm{U}_{2n}(q^{2})^{\varphi}\cong \mathrm{O}_{2n}^{-}(q)$.
\item[(ii)]  If $n$ is odd and $-1$ is a square in $\mathbb{F}_{q}$ then $\mathrm{U}_{2n}(q^{2})^{\varphi}\cong \mathrm{O}_{2n}^{+}(q)$.
\end{enumerate}
\end{enumerate}
\end{thm}

\begin{proof}  Let $M$ denote the common Gram matrix of $\beta$ and $\beta_{\varphi}$ with respect to the basis $\{h_{i}|i=1,\ldots,2n\}$ produced in Construction \ref{cons.suitablebasis}.  Then, with respect to this basis, a transformation $A\in\mathrm{GL}(V)$ lies in $\mathrm{U}_{2n}(q^{2})^{\varphi}$ if and only if both 
$$
A^{t}M\sigma(A)=M\mbox{ and }A^{t}MA=M.
$$

Since $M$ and $A$ are both invertible, it follows that $A=\sigma(A)$ and so $A$ is defined over $\mathbb{F}_{q}$.  Hence $\mathrm{U}_{2n}(q^{2})^{\varphi}$ consists of all $q$-rational matrices in $\mathrm{GL}(V)$ that satisfy $^{t}AMA=M$.

Since the matrix $M$ is also defined over $\mathbb{F}_{q}$, we see that the group of matrices satisfying $A^{t}MA=M$ is the full orthogonal group over $\mathbb{F}_{q}$ with respect to the symmetric bilinear form whose Gram matrix is $M$.  So now we must determine this group.  That is, $\mathrm{U}_{2n}(q^{2})^{\varphi}$ is isomorphic to either $\mathrm{O}_{2n}^{+}(q)$ or $\mathrm{O}_{2n}^{-}(q)$.

Recall that if $L_{1}$ and $L_{2}$ are $\beta_{\varphi}$-orthogonal elliptic lines, then $L_{1}\perp L_{2}$ can be written as $H_{1}\perp H_{2}$ where each $H_{i}$ is a hyperbolic line.  It follows that the isomorphism type of $\mathrm{U}_{2n}(q^{2})^{\varphi}$ is determined by the last two blocks $M_{n-1}$ and $M_{n}$ if $n$ is even, and the last block $M_{n}$ if $n$ is odd.

The rest of the proof follows easily by combining Lemma \ref{lem.determinetype} with the matrices from Construction \ref{cons.suitablebasis}.  The only subtlety is that Lemma \ref{lem.nsqnorm} is needed to show that if $\varphi$ is non-square type flip then $\det(M_{n})$ is a square in $\mathbb{F}_{q}$.
\end{proof}

\bibliographystyle{plain}
\bibliography{references}

\end{document}